\documentclass{amsart}

\usepackage{graphicx}
\usepackage{amsmath,amssymb,amsfonts}
\usepackage{multirow}
\usepackage{hhline}

\usepackage{xcolor}
\usepackage{amsmath}
\usepackage{mathrsfs,url}
\usepackage{mathtools}
 
\usepackage{appendix}
\usepackage{color}

\usepackage{bbm}

\usepackage{url}

\numberwithin{equation}{section}
\numberwithin{table}{section}
\numberwithin{figure}{section}
\gdef\AQ#1{}
\gdef\CQ#1{}

\newtheorem{theorem}{Theorem}[section]
\newtheorem{lemma}[theorem]{Lemma} 
\newtheorem{proposition}[theorem]{Proposition}

\newtheorem{alg}[theorem]{Algorithm} 
\theoremstyle{EX}
 
\newtheorem{example}[theorem]{Example} 
 
\newtheorem{definition}[theorem]{Definition} 
\newtheorem{question}[theorem]{Question}

\newcommand{\bdes}{\begin{description}}
\newcommand{\edes}{\end{description}}

\newcommand{\bal}{\begin{align}}
\newcommand{\eal}{\end{align}}

\newcommand{\bnum}{\begin{enumerate}}
\newcommand{\enum}{\end{enumerate}}

\newcommand{\bit}{\begin{itemize}}
\newcommand{\eit}{\end{itemize}}

\newcommand{\beq}{\begin{equation}}
\newcommand{\eeq}{\end{equation}}

\newcommand{\baray}{\begin{array}}
\newcommand{\earay}{\end{array}}

\newcommand{\bsry}{\begin{subarray}}
\newcommand{\esry}{\end{subarray}}

\newcommand{\bca}{\begin{cases}}
\newcommand{\eca}{\end{cases}}

\newcommand{\bcen}{\begin{center}}
\newcommand{\ecen}{\end{center}}

\newcommand{\bbm}{\begin{bmatrix}}
\newcommand{\ebm}{\end{bmatrix}}

\newcommand{\bmx}{\begin{matrix}}
\newcommand{\emx}{\end{matrix}}

\newcommand{\bpm}{\begin{pmatrix}}
\newcommand{\epm}{\end{pmatrix}}

\newcommand{\btab}{\begin{tabular}}
\newcommand{\etab}{\end{tabular}}

\newcommand{\N}{\mathbb{N}}

\newcommand{\cpx}{\mathbb{C}}
\newcommand{\lmd}{\lambda}

\newcommand{\af}{\alpha}
\newcommand{\bt}{\beta}
\newcommand{\reff}[1]{(\ref{#1})}
\newcommand{\mcal}[1]{\mathcal{#1}} %change \mc to \mcal in moor template
\newcommand{\cv}[1]{\mbox{conv}\left(#1\right)}

\newcommand{\supp}[1]{\mbox{supp}(#1)}
\newcommand{\st}{\mathit{s.t.}}
\newcommand{\ideal}{\mathit{Ideal}}
\newcommand{\qmod}{\mathit{Qmod}}
\newcommand{\tx}{\tilde{x}}
\newcommand{\ty}{\tilde{y}}
\newcommand{\tc}{\tilde{c}}
\newcommand{\sig}{\sigma}
\newcommand{\Sig}{\Sigma}
\newcommand{\Dt}{\Delta}
\newcommand{\hm}{\mathit{hom}}
\newcommand{\inte}[1]{\mathit{int}\left(#1\right)}
\newcommand{\cl}[1]{\mathit{cl}\left(#1\right)}
\newcommand{\tr}[1]{\mbox{trace}(#1)}
\newcommand{\mt}[1]{\mathtt{#1}}

\DeclareMathOperator{\rank}{rank}

\def\re{\mathbb{R}}
\def\cpx{\mathbb{C}}

\def\eps{\epsilon}

\def\nn{\nonumber}

\begin{document}

% \RUNAUTHOR{Nie and Yang}
\title{The Multi-Objective Polynomial Optimization}

% Department of Mathematics,
% Department of Electrical \& Computer Engineering,

% \HISTORY{}
\author{Jiawang Nie}
\address{Department of Mathematics,
University of California San Diego,
9500 Gilman Drive, La Jolla, CA, USA, 92093.}
\email{njw@math.ucsd.edu}

\author{Zi Yang}
\address{Department of Mathematics and Statistics,
University at Albany SUNY,
1400 Washington Ave, Albany, NY, USA, 12222.}
\email{zyang8@albany.edu}

\begin{abstract}
The multi-objective optimization is to optimize several objective functions
over a common feasible set. Since the objectives usually do not share
a common optimizer, people often consider (weakly) Pareto points.
This paper studies multi-objective optimization problems
that are given by polynomial functions.
First, we study the geometry for (weakly) Pareto values
and represent Pareto front as the boundary of a convex set.
Linear scalarization problems (LSPs)
and Chebyshev scalarization problems (CSPs)
are typical approaches for getting (weakly) Pareto points.
For LSPs, we show how to use tight relaxations to solve them,
how to detect existence or nonexistence of proper weights.
For CSPs, we show how to solve them by moment relaxations.
Moreover, we show how to check if a given point is a
(weakly) Pareto point or not and how to detect
existence or nonexistence of (weakly) Pareto points.
We also study how to detect unboundedness of polynomial optimization,
which is used to detect nonexistence of
proper weights or (weakly) Pareto points.
\end{abstract}


\maketitle

\section{Introduction}

The multi-objective optimization problem (MOP) is to optimize several
objectives simultaneously over a common feasible set.
MOPs have broad applications in economics \cite{geiger2011},
finance \cite{chen2017mean}, medical science \cite{rosen2017,van2012multi},
and machine learning \cite{wang2014cvx}.
In this paper, we consider the MOP in the form
\begin{equation} \label{prob:mop}
\left\{ \begin{array}{rl}
\min  &   f(x) \coloneqq (f_1(x),\ldots,f_m(x))  \\
\st &   c_i(x)=0 \, (i \in \mcal{E}),  \\
    &   c_j(x) \ge 0 \, (j \in \mcal{I}), \\
 \end{array} \right.
\end{equation}
where all functions $f_i, c_i, c_j$ are polynomials in $x  \coloneqq (x_1, \ldots, x_n) \in \re^n$.
The $\mcal{E}$ and $\mcal{I}$ are disjoint finite label sets.
Let $K$ denote the feasible set of \reff{prob:mop}.
Generally, there does not exist a point such that
all $f_i$'s are minimized simultaneously.
People often look for a point such that some or all of the objectives
cannot be further optimized. This leads to the following concepts
(see \cite{marler2004survey,miettinen2012nonlinear,vecopt2011}).

\begin{definition}
A point $x^*\in  K $ is said to be a Pareto point (PP) if there is no
$x \in  K$ such that $f_i(x) \le f_i(x^*)$ for all $i=1,\ldots,m$ and
$f_j(x) < f_j(x^*)$ for at least one $j$.
The point $x^*$ is said to be a weakly Pareto point (WPP)
if there is no $x \in K$ such that $f_i(x)< f_i(x^*)$ for all $i=1,\ldots,m$.
\end{definition}

In the literature,  Pareto points (resp., weakly Pareto points)
are also referenced as Pareto optimizers (resp., weakly Pareto optimizers),
or Pareto solutions (resp., weakly Pareto solutions).
A vector $v  \coloneqq  (v_1, \ldots, v_m)$ is called a
{\it Pareto value} (resp., weakly Pareto value)
for \reff{prob:mop} if there exists a Pareto point
(resp., weakly Pareto point) $x^*$
such that $v = f(x^*)$. {\it Pareto front}
is the set of objective values at Pareto points.
Every Pareto point is a weakly Pareto point,
while the converse is not necessarily true.
%
%For instance, consider the MOP
%\[
%%%\left\{ \begin{array}{rl}
%\min \quad (x_1,x_2) \quad
% \st \quad  0 \le x_1,x_2\le 1 .
%%%\end{array} \right.
%\]
%The point $(0,0)$ is the unique PP.
%However, each $(0,t)$ or $(t,0)$, with $0< t \le 1$,
%is a WPP but not a PP.
%
Detecting existence or nonexistence
of (weakly) Pareto points is a major task for MOPs. We refer to
%\cite{bao2010,bao2007variational,gutierrez2014,KimPhamTuy19}
\cite{bao2007variational,bao2010,KimPhamTuy19,
marler2004survey,miettinen2012nonlinear,vecopt2011}
for related work about existence of PPs and WPPs.

Scalarization is a classical method for finding PPs or WPPs.
It transforms a MOP into a single objective optimization problem.
A frequently used scalarization is
a nonnegative linear combination of objectives.

\begin{definition}\label{def:lsp}
The linear scalarization problem (LSP) for the MOP~\reff{prob:mop},
with a nonzero weight $w  \coloneqq  (w_1, \ldots, w_m) \ge 0$, is
\begin{equation} \label{prob:LSP1}
\baray{rl}
\min  &   w_1f_1(x) +\cdots + w_m f_m(x) \\
\st  &  x \in K .
\earay
\end{equation}
\end{definition}

{For the LSP \reff{prob:LSP1},
the optimization remains unchanged if we normalize the nonzero weight $w$
such that $\sum_{i=1}^m w_i=1,w_i\ge 0$.
For neatness of the paper, one can equivalently consider
nonzero and nonnegative weights for LSPs.}
Every minimizer of the LSP \reff{prob:LSP1}
is a weakly Pareto point for nonzero $w \ge 0$
and every minimizer is a Pareto point for $w > 0$.
Varying weights in \reff{prob:LSP1} may give different (weakly) Pareto points.
A nonzero weight $w$ is said to be {\it proper} if the LSP \reff{prob:LSP1}
is bounded below. Otherwise, the $w$ is called {\it improper}.
One wonders whether or not every Pareto point
is a minimizer of \reff{prob:LSP1} for some weight $w$.
However, this is sometimes not the case
(see \cite{fleming1986computer,zionts1988multiple}).
For instance, Example \ref{ex: CSP} has infinitely many Pareto points,
but only two of them can be obtained by solving LSPs.
Under some assumptions, LSPs may
give all Pareto points (see \cite{emmerich2018tutorial}).

Another frequently used scalarization is the Chebyshev scalarization.
It requires to use the minimum value of each objective.

\begin{definition}
The Chebyshev scalarization problem (CSP) for the MOP
\reff{prob:mop}, with a nonzero weight $w = (w_1, \ldots, w_m) \ge 0$, is
\begin{equation} \label{prob:CSP1}
\baray{rl}
\min  & \max\limits_{1 \le i \le m} \, w_i\big(f_i(x)- f_i^* \big)  \\
 \st  &  x \in K .
\earay
\end{equation}
where the minimum value $f_i^*  \coloneqq  \min\limits_{x\in  K } f_i(x) > -\infty$.
\end{definition}

Every minimizer of the CSP \reff{prob:CSP1}
is a weakly Pareto point. Interestingly, every weakly Pareto point
is the minimizer of a CSP for some weight
(see \cite{koski1987norm,miettinen2012nonlinear}).
However, the minimizer of a CSP may not be a Pareto point.
There also exist other scalarization methods,
such as the $\epsilon$-constraint method
\cite{anagnostopoulos2012online,matsatsinis2003agentallocator},
the lexicographic method \cite{clayton1982goal,jones2010practical}.
We refer to \cite{cho2017survey,donoso2016multi,marler2004survey,
miettinen2012nonlinear,ruiz1995characterization}
for different scalarizations.

%Evolutaional methods
% Beyond scalarizations, evolutionary methods are also frequently used
% for solving multi-objective optimization. Unlike scalarization methods,
% evolutionary methods usually approach the set of all Pareto points
% by using paradigms from natural evolution. We refer to \cite{deb2013evolutionary,emmerich2005emo,
% emmerich2018tutorial,zhang2007moea,zitzler1999multiobjective}
% for related work.

There exists important work for MOPs given by polynomials.
When all functions are linear,
a semidefinite programming method is given
to obtain the set of Pareto points in \cite{blanco2014}.
When the functions are convex polynomials,
Moment-SOS relaxation methods are given to
compute (weakly) Pareto points in \cite{jiao2020,JLZ20,JiaoLee21,JLS21,LeeJiao18},
as well as some useful conditions for existence of (weakly) Pareto points.
{
Since the Pareto front is an image set of polynomial functions,
semidefinite relaxations can be used to approximate the Pareto front,
as in the work \cite{MDJ14,MDJ15}.
%
%The work \cite{MDJ15} considers approximating general polynomial images
%of compact semialgebraic sets by semidefinite relaxations
%which can also be used to approximate the Pareto front.
%
}

When the functions are nonconvex polynomials,
nonemptiness and boundedness of Pareto solution sets
are shown in \cite{LHF21}, under certain regularity conditions.
When the objectives are polynomials and $K$ is the entire space $\re^n$,
some novel conditions are shown for existence of
(weakly) Pareto points in \cite{KimPhamTuy19}.
%
%{However, those conditions are hard to be implemented numerically.}
%
The following questions are of great {interest} for studying MOPs:

\bit

\item What is a convenient description for the set of (weakly) Pareto values?
How can we represent the Pareto front
in a geometrically clean way?

\item For an LSP, how can we solve it efficiently for a Pareto point?
% If an LSP is unbounded below, how can we detect the unboundedness?
When the contraint $K$ is unbounded, how can we find a proper weight
such that the LSP is bounded?
How can we detect nonexistence of proper weights?

\item For a CSP, how can we solve it efficiently for a weakly Pareto point?
How do we get the global minimum value for each objective?
If some minimum value is $-\infty$,
how can we get a weakly Pareto point?

\item For a given point, how can we detect if it
is a (weakly) Pareto point?
How can we get a (weakly) Pareto point if LSPs/CSPs fail to give one?
How do we detect nonexistence of (weakly) Pareto points?

\eit

\subsection*{Contributions}

The above questions are the major topics of this paper.
When MOPs are given by polynomials,
there are special properties for them.
The following are our major contributions.

We study the convex geometry for (weakly) Pareto values.
The epigraph set, i.e., the set $\mcal{U}$ as in \reff{set:U},
is useful for (weakly) Pareto values.
We give a characterization for the Pareto front.
When the objectives are convex,
we show that the set of weakly Pareto values
can be expressed in terms of the boundary of a convex set.
When the MOP is given by SOS convex polynomials,
we show that $\mcal{U}$ can be given by semidefinite representations.
This is shown in Section~\ref{sc:char}.
%
%the set defined in Proposition \reff{prop:cvx LSP}.
%(Weakly) Pareto values can be described by supporting hyperplanes of the set $\mcal{U}$.
%The convex representation \reff{convex:notPV}
%is proposed to characterize the interior of $\cv{\mcal{U}}$.
%Specifically, the set of weakly Pareto values can be obtained by
%solving the SDP problem \reff{weakPV:SOSCVX} under some SOS convexity assumptions.
%

For solving LSPs and CSPs, or detecting nonexistence of
(weakly) Pareto points, we often need to detect
whether or not an optimization problem is unbounded.
There exists few work for
detecting unboundedness in nonconvex optimization.
We give a convex relaxation method for
detecting unboundedness in polynomial optimization
under some genericity assumptions.
To the best of the authors' knowledge,
this is the first work that can achieve this goal.
The results are in Section~\ref{sc:unbounded}.

We discuss how to solve LSPs in Section~\ref{sc:LSP}.
Under a genericity assumption, we give a tight relaxation method for
solving LSPs and obtaining Pareto points.
When the feasible set $K$ is unbounded,
we show how to find proper weights such that the LSP is bounded below.
We also show how to detect that the LSP is unbounded below for all weights,
i.e., how to detect nonexistence of proper weights.

Section~\ref{sc:CSP} studies how to solve CSPs.
We first apply the tight relaxation method to compute global
minimum values $f_1^*, \ldots, f_m^*$ for the individual objectives.
After that, we formulate the CSP
equivalently as a polynomial optimization problem
and then solve it by using Moment-SOS relaxations.

Section~\ref{sc:detect} discusses how to detect if a given point
is a (weakly) Pareto point or not. This can be done
by solving certain polynomial optimization.
We also show how to detect existence or nonexistence of (weakly) Pareto points.
This requires to solve some moment feasibility problems.

We make some conclusions and propose some open questions in Section~\ref{sc:con}.
Section~\ref{sc:pre} reviews some basic results
for optimization with polynomials and moments.

\section{Preliminary}
\label{sc:pre}

\subsection*{Notation}
The symbol $\N$ (resp., $\re$, $\cpx$) denotes the set of
nonnegative integral (resp., real, complex) numbers.
The $\re^n_+$ stands for the nonnegative orthant, i.e.,
the set of nonnegative vectors.
For each label $i$, the $e_i$ denotes the vector of all zeros
excepts its $i$th entry being $1$,
while $e$ denotes the vector of all ones.
For an integer $k>0$, denote $[k]  \coloneqq  \{1,2,\ldots, k\}.$
For $t \in \re$, $\lceil t \rceil$ denotes
the smallest integer greater than or equal to $t$.
Denote by $\re[x]  \coloneqq  \re[x_1,\ldots,x_n]$
the ring of polynomials in $x \coloneqq (x_1,\ldots,x_n)$
with real coefficients. The $\re[x]_d$
stands for the set of polynomials in $\re[x]$
with degrees at most $d$.
For a polynomial $p$, $\deg(p)$ denotes its total degree,
$\tilde{p}$ denotes its homogenization, and $p^{hom} $
denotes the homogeneous part of the highest degree.
For $\af \coloneqq  (\af_1, \ldots, \af_n) \in \N^n$,
we denote $x^\af \,  \coloneqq  \, x_1^{\af_1} \cdots x_n^{\af_n}$
and $|\af| \coloneqq \af_1+\cdots+\af_n$.
The power set of degree $d$ is
\[
\N^n_d  \coloneqq \{ \af  \in \N^n \mid |\af|  \leq d \}.
\]
The vector of monomials in $x$ and up to degree $d$ is
\[
[x]_{d}  \coloneqq  \bbm 1 & x_1 &\cdots & x_n & x_1^2 & x_1x_2 & \cdots & x_n^{d}\ebm^T.
\]
The superscript $^T$ denotes the transpose of a matrix/vector.
The $I_N$ stands for the $N$-by-$N$ identity matrix.
By writing $X\succeq 0$ (resp., $X\succ 0$), we mean that
$X$ is a symmetric positive semidefinite
(resp., positive definite) matrix.
For a set $T$, $\cv{T}$ denotes its convex hull,
$\cl{T}$ denotes its closure, and
$\inte{T}$ denotes its interior, under the Euclidean topology.
The cardinality of $T$ is $|T|$.
For a vector $u$, the $\|u\|$ denotes its standard Euclidean norm.
%
%In particular, for a vector $a$,
%$\diag(a)$ denotes the diagonal matrix whose
%diagonal entry vector is $a$.
%For a function $f$ in $x$, $f_{x_i}$ denotes the partial derivative
%of $f$ with respect to $x_i$.
%
For a function $h$ in $x$, the $\nabla h$ denotes
its gradient vector in $x$.
All computational results are shown with four decimal digits.

%
%For a set $T \subseteq \re^n$, the image of the vector function
%$f$ on $T$ is
%\[
%f(T) \,  \coloneqq  \, \{ (f_1(x), \ldots, f_m(x)): \, x \in T \}.
%\]
%

\subsection{Positive polynomials}
\label{ssc:posp}

A subset $I \subseteq \re[x]$ is an ideal if
$ I \cdot \re[x] \subseteq I$ and $I+I \subseteq I$. For a tuple
$p=(p_1,\ldots,p_k)$ of polynomials in $\re[x]$,
$\mbox{Ideal}(p)$ denotes the smallest ideal containing all $p_i$,
which is the set
$
p_1 \cdot \re[x] + \cdots + p_k  \cdot \re[x].
$
In computation, we often need to work with the {\it truncation} of degree $2k$:
\[
\ideal[p]_{2k}  \,  \coloneqq  \,
p_1 \cdot \re[x]_{2k-\deg(p_1)} + \cdots + p_k  \cdot \re[x]_{2k-\deg(p_k)}.
\]
A polynomial $\sig$ is said to be a sum of squares (SOS)
if $\sig = s_1^2+\cdots+ s_k^2$ for some polynomials $s_1,\ldots, s_k$.
Checking if a polynomial is SOS can be done by solving
a semidefinite program (SDP) \cite{Las01}.
If a polynomial is SOS, then it is nonnegative everywhere.
The set of all SOS polynomials in $x$ is denoted by $\Sig[x]$
and its $d$th truncation is
$
\Sig[x]_d  \coloneqq  \Sig[x] \cap \re[x]_d.
$
For a tuple $q=(q_1,\ldots,q_t)$ of polynomials,
its {\it quadratic module} is
\[
\qmod[q] \,  \coloneqq  \,  \Sig[x] + q_1 \cdot \Sig[x] + \cdots + q_t \cdot \Sig[x].
\]
The truncation of degree $2k$ for $\qmod[q]$ is
\[
\qmod[q]_{2k} \,  \coloneqq  \,
\Sig[x]_{2k} + q_1 \cdot \Sig[x]_{2k - \deg(g_1)}
+ \cdots + q_t \cdot \Sig[x]_{2k - \deg(q_t)}.
\]

A subset $A \subseteq \re[x]$ is said to be {\it archimedean}
if there exists $\sig \in A$ such that
$\sig(x) \geq 0$ defines a compact set in $\re^n $.
If $\ideal[p]+\qmod[q]$ is archimedean, then the set
$T  \coloneqq  \{x \in \re^n: p(x)=0, \, q(x)\geq 0 \}$ must be compact.
The reverse is not necessarily true.
However, if $T$ is compact, the archimedeanness
can be met by adding a redundant ball condition.
When $\ideal[p]+\qmod[q]$ is archimedean,
every polynomial that is positive on $T$
must belong to $\ideal[p]+\qmod[q]$.
This conclusion is referenced as Putinar's Positivstellensatz~\cite{Put}.
Furthermore, if a polynomial is nonnegative on $T$,
then it also belongs to $\ideal[p]+\qmod[q]$,
under some standard optimality conditions on its minimizers (see \cite{Nie-opcd}).

\subsection{Localizing and moment matrices}

Denote by $\re^{\N_d^n}$ the space of real sequences labeled by $\af \in \N_d^n$.
A vector $y  \coloneqq (y_\af)_{ \af \in \N_d^n}$ is called a
{\it truncated multi-sequence} (tms) of degree $d$.
It gives a linear functional on $\re[x]_d$ such as
\begin{equation}
\langle \sum\limits_{\af \in \N_d^n} f_\af x^\af , y \rangle
\,  \coloneqq   \, \sum\limits_{\af \in \N_d^n}  f_\af y_\af,
\end{equation}
where each $f_\af$ is a coefficient.
The tms $y$ is said to {\it admit} a {\it Borel measure} $\mu$ if
$y_\af = \int x^\af \mathtt{d} \mu$ for all $\af \in \N_d^n$.
If it exists, such $\mu$ is called a
{\it representing measure} for $y$ and $y$ is said to admit the measure $\mu$.
The support of $\mu$ is denoted as $\supp{\mu}$.
If the cardinality $|\supp{\mu}|$ is finite,
the measure $\mu$ is called {\it finitely atomic}.
It is called {\it $r$-atomic} if $|\supp{\mu}|=r$.

In optimization, the support of $\mu$ is often constrained in a set $K$.
For a degree $d$, denote the moment cone
\beq \label{set:Rd(K)}
\mathscr{R}_d(K)  \,  \coloneqq  \,
\Big\{ y \in \re^{ \N^n_d }:  \,
\exists \, \mu, \,  \, y = \int [x]_d \mathtt{d} \mu, \,
\supp{\mu} \subseteq K   \Big\}.
\eeq
The dual cone of $\mathscr{R}_d(K)$ is the nonnegative polynomial cone
\beq  \label{set:Pd(K)}
\mathscr{P}_d(K)  \,  \coloneqq  \,
\Big\{ p \in \re[x]_d: \, p(x) \ge 0 \, \forall \,x \in  K   \Big\}.
\eeq
The dual cone of $\mathscr{P}_d(K)$ is the closure of
$\mathscr{R}_d(K)$. When $K$ is compact,
the moment cone $\mathscr{R}_d(K)$ is closed.
We refer to \cite{LasBk15,Lau09}
for more details about moment cones.

Consider a polynomial $q \in \re[x]_{2k}$ with $\deg(q) \leq 2k$.
The $k$th {\it localizing matrix} of $q$,
generated by a tms $z \in \re^{\N^n_{2k}}$,
is the symmetric matrix $L_q^{(k)}[z]$ such that
\begin{equation}  \label{LocMat}
vec(a_1)^T \big( L_q^{(k)}[z] \big) vec(a_2)  \,=\,
\langle  q a_1 a_2, z \rangle
\end{equation}
for all $a_1,a_2 \in \re[x]_{k - \lceil \deg(q)/2 \rceil}$.
(The $vec(a_i)$ denotes the coefficient vector of $a_i$.)
% For instance, when $n=2$ and $k=2$ and $q = 1 - x_1^2-x_2^2$, we have
% \[
% L_q^{(2)}[w]=\left [
% \begin{matrix}
% z_{00}-z_{20}-z_{02} &  z_{10}-z_{30}-z_{12} &  z_{01}-z_{21}-z_{03} \\
% z_{10}-z_{30}-z_{12} &  z_{20}-z_{40}-z_{22} &  z_{11}-z_{31}-z_{13} \\
% z_{01}-z_{21}-z_{03} &  z_{11}-z_{31}-z_{13} &  z_{02}-z_{22}-z_{04} \\
% \end{matrix}\right ].
% \]
When $q = 1$, $L_q^{(k)}[z]$ is called a {\it moment matrix} and we denote
\[
M_k[z] \, \coloneqq  \, L_{1}^{(k)}[z].
\]
The columns and rows of $L_q^{(k)}[z]$, as well as $M_k[z]$,
are labeled by $\af \in \N^n$ with $2|\af|  \leq 2k - \deg(q) $.
%
%When $q = (q_1, \ldots, q_t)$ is a tuple of polynomials,
%then we define
%\begin{equation}  \label{block:LM}
%L_q^{(k)}[z] \,  \coloneqq  \, \mbox{diag}
%\big( L_{q_1}^{(k)}[z], \ldots,  L_{q_t}^{(k)}[z] \big),
%\end{equation}
%which is a block diagonal matrix.
%%
%Moment and localizing matrices can be used to construct
%semidefinite programming relaxations for solving polynomial optimization.
%%
%We refer to \cite{Todd} for a survey on semidefinite programs.
%

Each $y \in \mathscr{R}_d(K)$ can be extended to
a tms $z  \in \mathscr{R}_{2t}(K)$ such that $y = z|_d$,
where $d \le 2t$ and $z|_d$ denotes the truncation of $z$ with degree $d$:
\beq \label{truncation:z}
z|_d \,  \coloneqq  \, (z_\af)_{ \af \in \N^n_d }.
\eeq
%
%Therefore, if $y \in \mathscr{R}_d(K)$, for the above $z$, it holds
%$
%\mathscr{L}_y \big( c_i p^2  \big) = 0 \, (i \in \mcal{E}), \,
%\mathscr{L}_y \big( c_j p^2  \big) \succeq 0 \, (j \in \mcal{I})
%$, for all $p$ with $\deg(c_i p^2), \deg(c_j p^2) \le 2t$.
%
When $K$ is the feasible set of \reff{prob:mop},
a necessary condition for $z \in \mathscr{R}_{2t}(K)$ is
\[
L_{c_i}^{(t)}[z] = 0 \,  (i \in \mcal{E}), \quad
L_{c_j}^{(t)}[z] \succeq 0 \,  (j \in \mcal{I}),
\]
while they may not be sufficient (see \cite{LasBk15,Lau09}).
However, if $z$ further satisfies
\begin{equation} \label{cond:FEC}
\rank \, M_{t-d_c}[z] \,= \,\rank \, M_t[z],  \qquad
\end{equation}
then $z$ admits a $r$-atomic measure supported in $K$,
with $r = \rank \, M_t[z]$. The above integer $d_c$ is the degree
\begin{equation} \label{deg:dc}
d_c \,  \coloneqq  \,\max \{ \lceil \deg(c_i)/2 \rceil :
i \in \mcal{E} \cup \mcal{I} \}.
\end{equation}
This condition \reff{cond:FEC} is called flat extension
(see \cite{CurFia96,CurFia05,HenLas05,Lau05}).
To get optimizers in computation,
the flat truncation is more frequently used (see \cite{nie2013certifying}).

Moment and localizing matrices are important tools for
solving polynomial optimization \cite{FNZ18,HenLas05,Las01,PMISDr}.
They are also useful in tensor decompositions \cite{NieGP17,NieZhang18}.
We refer to \cite{LasBk15,LasICM,Lau09,LauICM}
for the books and surveys about polynomial and moment optimization.

\section{Geometry of Pareto values}
\label{sc:char}

Recall that
a vector $v  \coloneqq  (v_1, \ldots, v_m)$ is a {\it Pareto value} (PV) if
there exists a Pareto point $x^*$ such that $v = f(x^*)$.
Similarly, $v$ is called a weakly Pareto value (WPV)
if $v = f(p)$ for a weakly Pareto point $p$.
PVs and WPVs are closely related to the epigraph set
\beq \label{set:U}
\mcal{U} \, \coloneqq  \, \{ u =(u_1, \ldots, u_m) \mid u_i \ge f_i(x),
\, \mbox{for some} \, x \in  K  \}.
\eeq
The image of the set $K$ under the objective vector $f=(f_1,\ldots,f_m)$ is
\[
f(K) \,  \coloneqq  \, \{ (f_1(x), \ldots, f_m(x)): \, x \in K \}.
\]
Then, $\mcal{U} \,= \, f(K) + \re_+^m$ and its convex hull
$\cv{ \mcal{U} }  = \cv{ f(K)  } + \re_+^m$.
If $K$ is convex and each objective $f_i$ is convex,
the set $\mcal{U}$ is also convex.
The converse is not necessarily true.
When $\mcal{U}$ is convex, every Pareto point
is a minimizer of some LSP (see \cite{emmerich2018tutorial}).
In this section, we study the geometry of PVs and discuss
how to characterize PVs and WPVs through the set $\mcal{U}$.

\subsection{Supporting hyperplanes}
\label{ssc:support}

For a nonzero vector $w \in \re^m$ and $b \in \re$, the set
\[
H  \, = \,  \{u \in \re^m: w^Tu = b\}
\]
is a supporting hyperplane for $\mcal{U}$ if
%if there is a point $v \in \mcal{U}$ such that
%$w^Tu \ge b=w^Tv$ for all $u \in \mcal{U}$.
$
b = \inf_{u \in \mcal{U} }\, w^Tu.
$
The $w$ is the normal of $H$.
In particular, if there exists $v \in \mcal{U}$ such that
$w^Tu \ge w^Tv$ for all $u \in \mcal{U}$,
then $H$ is called a supporting hyperplane through $v$.
Since $\mcal{U}$ contains $f(x)+ \re_+^n$,
the normal $w$ must be nonnegative,
for $H$ to be a supporting hyperplane.

%
%{remove $\pi$-minimal related contents?
%They are not very related to our main contributions
%and may cause extra confusions to readers.
%I am also thinking if there is some convenient way
%to characterize the same thing.}
%

In MOP, people often use different orderings to define various minimizers.
We refer to \cite{marler2004survey,miettinen2012nonlinear,vecopt2011}
for general orderings in MOP. Here we introduce
the convenient lexicographical ordering, up to permutations.
%%To characterize Pareto values, we define a kind of {\it minimal points}.
Let $\pi$ be a permutation of $(1,\ldots, m)$.
For a set $T \subseteq \re^m$, construct the following chain of nesting subsets
\[
T = T_0 \supseteq T_1 \supseteq \cdots \supseteq T_m
\]
such that: for each $k = 1, \ldots, m$,
$T_k$ is the subset of vectors in $T_{k-1}$
whose $\pi(k)$th entry is the smallest.
If $T_m \ne \emptyset$, then each $v \in T_m$
is called a {\it $\pi$-minimal} point of $T$.
For $u, v \in T_m$, all the entries of $u,v$ must be the same,
so $u=v$ and hence $T_m$ consists of a single point,
if it is nonempty. In particular, if $T$ is compact,
then $T_m \ne \emptyset$ and it consists of a single point.

PVs and WPVs are characterized in the following.
Some of these results may already exist in the literature.
%%\cite{marler2004survey,miettinen2012nonlinear,vecopt2011}.
For convenience of readers, we summarize them together
and give direct proofs.

\begin{proposition} \label{pro:char:PV}
Let $\mcal{U}$ be as in $\reff{set:U}$.
For each $v \in f(K)$, we have:

\bit

\item [(i)]
The vector $v$ is a WPV if and only if
$v$ lies on the boundary of $\mcal{U}$.
Moreover, if $v$ is an extreme point of $\cv{\mcal{U}}$,
then $v$ is a PV.

\item [(ii)] Assume $\mcal{U}$ is convex.
If $v$ is a WPV, then there exists a supporting hyperplane for $\mcal{U}$
through $v$ whose normal is nonnegative, i.e., there exists
$0 \ne w \ge 0$ such that $w^Tu \ge w^Tv$ for all $u \in \mcal{U}$.

\item [(iii)]
Suppose $H=\{u: w^Tu = w^Tv \}$ is a supporting hyperplane for $\mcal{U}$
through $v$, with a normal vector $0 \ne w \ge 0$.
If $w > 0$, then $v$ is a PV. For $w$ with a zero entry,
if $u\in f(K)$ is a $\pi$-minimal point of $H\cap \mcal{U}$, then $u$ is a PV.
If $u \in f(K)$ is an extreme point of $H \cap \mcal{U}$,
then $u$ is also a PV.

\eit
\end{proposition}
\proof
(i) If $v$ lies on the boundary of $\mcal{U}$,
then there is no $p \in K$ such that
$f(p) < v$, so $v$ is a WPV.
If $v$ is an interior point of $\mcal{U}$,
then exist $p \in  K $ and $q \ge 0$ such that
$f(p) + q < v$, which denies that $v$ is a WPV.
This shows that $v$ is a WPV if and only if
$v$ lie on the boundary of $\mcal{U}$.

Next, suppose $v$ is an extreme point of $\cv{\mcal{U}}$.
Suppose otherwise that $v$ is not a PV,
then there exists $p \in  K$ such that
$
f(p) \le v,  f(p) \ne v.
$
This means that $v = f(p) + q$, for some $0 \ne q \in \re_+^m$. Hence
$
v = \frac{1}{2} f(p) + \frac{1}{2} ( f(p) + 2q),
$
which implies $v$ is not an extreme point of $\cv{\mcal{U}}$, a contradiction.
So $v$ is a PV.

(ii) If $v$ is a WPV, then $v$ lies on the boundary of $\mcal{U}$.
Since $\mcal{U}$ is convex, there is a supporting hyperplane for $\mcal{U}$
through $v$, i.e., there exists
$w \ne 0$ such that $w^Tu \ge w^Tv$ for all $u \in \mcal{U}$.
The set $\mcal{U}$ contains $v + \re_+^m$, so $w \ge 0$.

(iii) For the case $w>0$, the conclusion is obvious.
When $w$ has zero entries,
let $I = \{i \in [m]: w_i > 0 \}$.
To prove $u  \coloneqq  (u_1, \ldots, u_m)$ is a PV,
suppose $p \in K$ is such that
$
f(p) \le u .
$
Since $u \in H \cap \mcal{U}$,
$
w^T f(p) \le w^Tu = w^T v.
$
Also note that $w^T f(p) \ge w^T u$, since $H$ is a supporting hyperplane.
So we must have $w^T f(p) = w^T v$ and $f_i(p) = u_i$ for all $i \in I$.
Write that $u = f(p) + q$, for some $q \in \re_+^m$.
Note that $q_i = 0$ for all $i \in I$.
Since $u$ is a $\pi$-minimal point of $H \cap \mcal{U}$ and $f(p) \le u$,
the vector $f(p)$ is also a $\pi$-minimal point of $H \cap \mcal{U}$.
Hence $u = f(p)$, by the $\pi$-minimality.
This means that $u$ is a PV.

When $u$ is an extreme point of $H \cap \mcal{U}$,
we can prove that $u$ is a PV in the same way as for the item (i).
% \Halmos
\endproof

We have the following remarks for Proposition~\ref{pro:char:PV}.

\bit

\item Not every WPV lies on the boundary of
$\cv{\mcal{U}}$.  For instance, consider
\beq \nn
\left\{ \baray{rl}
\min & (x_1, \, x_2) \\
\st &  x_1 \ge 0, x_2 \ge 0, \, x_1^2 + x_2^2 = 1.
\earay \right.
\eeq
For each $t \in (0,1)$, the point $(t, \sqrt{1-t^2})$ is a WPP (also a PP),
but it {does not lie} on the boundary of $\cv{\mcal{U}}$.

\item If $\mcal{U}$ is not convex, there may not exist a supporting hyperplane
through a WPV. For instance, in the above MOP,
for every $t \in (0,1)$, there is no supporting hyperplane
for $\cv{\mcal{U}}$ through $(t, \sqrt{1-t^2})$.

\item For the item (iii) of Proposition~\ref{pro:char:PV},
if $w$ has a zero entry, then $v$ may not be a Pareto value.
For instance, consider the unconstrained MOP
\[
%\left\{ \baray{rl}
\min \quad (x_1, \, x_2^2) .
%%\st \quad  x_1 \in \re^1, \, x_2 \in \re^1.
%\earay \right.
\]
For $w = (0, 1)$ and $v = (0, 0)$,
the equation $w^Tu = 0$ gives a supporting hyperplane through $(0,0)$,
but $(0,0)$ is not a Pareto value.

\item If $v$ is a PV, it may not be an extreme point of
$\mcal{U}$ or $H \cap \mcal{U}$.
For instance, consider the MOP
\beq \nn
\left\{ \baray{rl}
\min & (x_1, \, x_2) \\
\st &  x_1 \ge 0, x_2 \ge 0, \, x_1 + x_2 = 1.
\earay \right.
\eeq
The set $\mcal{U} = \{ x_1 \ge 0, \, x_2 \ge 0, \, x_1 + x_2 \ge 1 \}$.
Clearly, for every $t \in (0,1)$, the vector $(t, 1-t)$
is a PV, but it is not an extreme point of $\mcal{U}$.
The hyperplane $H =\{x_1 + x_2 = 1\}$ supports $\mcal{U}$ at $(t, 1-t)$.
However,  $(t, 1-t)$ is not an extreme point of the intersection
$H \cap \mcal{U}$, for every $t \in (0,1)$.

\eit

\subsection{A convex representation}
\label{ssc:cvxRep}

When the feasible set $K$ is bounded,
there always exist supporting hyperplanes for $\mcal{U}$.
When $K$ is unbounded, they may or may not exist.
For given $v=(v_1, \ldots, v_m) \in f(K)$,
how do we determine if there is a supporting hyperplane through it?
For this purpose, we consider the linear optimization in $w_0 \in \re$
and $w = (w_1, \ldots, w_m) \in \re^m$:
\beq  \label{maxw0:nng}
\left\{ \baray{rl}
\omega^*  \coloneqq  \max  & w_0 \\
\st   & 1 - e^T w = 0, \,  w_i \ge w_0 \, ( i \in [m]), \\
    & \sum\limits_{i=1}^m w_i( f_i(x) - v_i) \ge 0 \quad \mbox{on} \quad K.
\earay \right.
\eeq
Clearly, there is a supporting hyperplane through $v$
if and only if the optimal value $\omega^* \ge 0$.
Let $d$ be the maximum degree of objectives $f_i$.
The third constraint in \reff{maxw0:nng} is equivalent to the membership
\[
\baray{c}
\sum\limits_{i=1}^m w_i( f_i(x) - v_i) \in \mathscr{P}_d(K),
\earay
\]
where $\mathscr{P}_d(K)$ is the nonnegative
polynomial cone as in \reff{set:Pd(K)}.
The dual cone of $\mathscr{P}_d(K)$ is the closure
$\mbox{cl}\big( \mathscr{R}_d(K) \big)$,
where $\mathscr{R}_d(K)$ is the moment cone as in \reff{set:Rd(K)}.
%%%%%%%%%%%%%%%%%%%%%%%%%%%%%%%%%%%%%%%%%%%%%%%%%%%%%%
\iffalse

The Lagrange function for \reff{maxw0:nng} is
\begin{eqnarray*}
  \mcal{L}(x, \lmd, t, y) &=& w_0 + \sum\limits_{i=1}^m \lmd_i ( w_i-w_0) +
t (1 - e^T w ) + \langle  \sum_i w_i( f_i - v_i), y \rangle \\
&=&t + w_0(1-e^T\lmd) + \sum\limits_{i=1}^m w_i
\big( \lmd_i +\langle f_i-v_i, y \rangle -t \big),
\end{eqnarray*}

where $\lmd =(\lmd_1, \ldots, \lmd_m) \ge 0$, $t \in \re$,
and $y \in \mathscr{R}_d(K)$. The tms $y$ is labeled as
\[ y \, = \, (y_\af)_{ \af \in \N^n_d } . \]

\fi
%%%%%%%%%%%%%%%%%%%%%%%%%%%%%%%%%%%%%%%%%%%%%%%%%%%%%%%%%%%
The dual optimization of \reff{maxw0:nng} can be shown to be
\beq \label{min:mu:yinRd}
\left\{ \baray{rl}
\min  & t \\
\st & t - \langle f_i-v_i, y \rangle \ge 0 \, (i \in [m]), \\
    & 1 = m t - \sum\limits_{i=1}^m \langle f_i-v_i, y \rangle , \,
    \, y \in \mbox{cl} \big( \mathscr{R}_d(K) \big).
\earay \right.
\eeq
In the above, the vector $y$ is a tms labeled as
\[ y \, = \, (y_\af)_{ \af \in \N^n_d } . \]
If \reff{min:mu:yinRd} has a feasible point with $t < 0$,
then there are no nonnegative supporting hyperplanes through $v$.
Since each $v_i$ is a scalar, one can see that
\[
\langle f_i-v_i, y \rangle =
\langle f_i, y \rangle - v_i \langle 1, y_0 \rangle
= \langle f_i, y \rangle - v_i y_0.
\]
When $t<0$ is feasible for \reff{min:mu:yinRd},
there also exists a feasible $y \in \mathscr{R}_d(K)$ with $y_0 > 0$.
One can scale such $(t,y)$ so that $y_0=1$.
Hence, the existence of $t<0$ in \reff{min:mu:yinRd} is equivalent to
\beq \nn
\left\{ \baray{l}
 \tau  = m t' - \sum\limits_{i=1}^m (\langle f_i, y \rangle -v_i),  \\
 t' \ge \langle f_i, y \rangle - v_i , \, i =1,\ldots, m,  \\
\tau >0 > t', \, y_0 = 1, \,  y \in \mathscr{R}_d(K).
\earay \right.
\eeq
The above is then equivalent to that
\[
\left\{ \baray{l}
v_i  >  \langle f_i, y \rangle, \, i =1,\ldots, m,  \\
y_0 = 1,  \, y \in \mathscr{R}_d(K).
\earay \right.
\]
We define the set $\mcal{V}$ containing all such $v$:
\beq  \label{convex:notPV}
\mcal{V} \,  \coloneqq  \,
\left\{ v \left| \baray{l}
v = (v_1, \ldots, v_m) \\
v_i  >  \langle f_i, y \rangle, \, i =1,\ldots, m,  \\
y_0 = 1,  \, y \in \mathscr{R}_d(K)
\earay \right. \right\}.
\eeq
%
%A vector $v = (v_1, \ldots, v_m)$ is called a lower Pareto value (lPV) if
%there exists a point $u$ on the boundary of $\mcal{U}$
%such that $v \le u$. If $v$ is not a lower value,
%we call it an upper Pareto value.
%

\begin{theorem} \label{th:intU:WPV}
Assume $K$ has nonempty interior.
Then, the interior of the convex hull $\cv{ \mcal{U} }$
is the set $\mcal{V}$ as in \reff{convex:notPV}.
Moreover, when $\mcal{U}$ is convex, a vector $v \in f(K)$
is a weakly Pareto value if and only if $v$
belongs to the boundary of the closure $\mbox{cl}\big( \mcal{V} \big)$.
\end{theorem}
\proof
Since $K$ has nonempty interior, the cone $\mathscr{R}_d(K)$
has nonempty interior. Hence, the strong duality holds
between \reff{maxw0:nng} and \reff{min:mu:yinRd},
since \reff{min:mu:yinRd} has strictly feasible points.
This is because one can select $y$ from the interior of $\mathscr{R}_d(K)$,
choose $t$ sufficiently large to satisfy all the inequalities,
and then scale such $(t,y)$ for the equality to hold.

A point $v$ lies in the interior of
$\cv{ \mcal{U} }$ if and only if there is no supporting hyperplane
for $\mcal{U}$ through it. The normal of every
supporting hyperplane for $\mcal{U}$ is nonnegative.
Thus,  $v$ lies in the interior of $\cv{ \mcal{U} }$ if and only if
the optimal value $\omega^*$ of \reff{maxw0:nng}
is negative or it is infeasible. By the strong duality
between \reff{maxw0:nng} and \reff{min:mu:yinRd},
this is equivalent to that $v$ belongs to $\mcal{V}$.

When $\mcal{U}$ is convex, i.e., $\mbox{conv}(\mcal{U}) = \mcal{U}$,
a vector $v \in f(K)$ is a WPV if and only if $v$
lies on the boundary of $\mcal{U}$, by Proposition~\ref{pro:char:PV}.
This is equivalent to that
$v$ lies on the boundary of $\mbox{cl}\big( \mcal{V} \big)$,
since the interior of $\mcal{U}$ is $\mcal{V}$.
% \Halmos
 \endproof

%%%%%%%%%%%%%%%%%%%%%%%%%%%%%%%%%%%%%%%%%%%%%%%%%%%%%%%%%%%%%%%%
%%%%%%%%%%%%%%%%%%%%%%%%%%%%%%%%%%%%%%%%%%%%%%%%%%%%%%%%%%%%%%%%
\iffalse

The following is an example for Theorem~\ref{th:intU:WPV}.

\begin{example}\label{ex:notPV}
Consider the objectives
\[
\baray{l}
 f_1 = x_1^2x_2^2+x_2^2x_3^2+x_3^2x_4^2-x_1x_2^2x_3-x_2x_3^2x_4, \\
 f_2 = \sum\limits_{i=1}^4 x_i^4 - x_1^2(x_2+x_3)-x_2^2(x_3+x_4),
\earay
\]
and the constraints $x_1\ge 0, \ldots, x_4 \ge 0$. For given $v=(v_1,v_2)$,
we can solve a linear optimization problem with the convex set $\mcal{V}$,
to determine if $v$ belongs to $\cv{\mathcal{U}}$ or not.
The set $\cv{\mathcal{U}}$ is shown as the gray part in Figure~\ref{fig:notPV}.
The bold curve is a part of the boundary.
\end{example}

\begin{figure}
\includegraphics[width = 0.6\textwidth]{soscvx_Ex5_gray.png}
\caption{ The set $\cv{\mathcal{U}}$ for Example~\ref{ex:notPV}. }
\label{fig:notPV}
\end{figure}

\fi
%%%%%%%%%%%%%%%%%%%%%%%%%%%%%%%%%%%%%%%%%%%%%%%%%%%%%%%%%%%%%%%%
%%%%%%%%%%%%%%%%%%%%%%%%%%%%%%%%%%%%%%%%%%%%%%%%%%%%%%%%%%%%%%%%

A computational efficient description for the moment cone
$\mathscr{R}_d(K)$ is usually not available.
However, when the polynomials are SOS-convex,
there exists a semidefinite representation for the set $\mcal{V}$
in \reff{convex:notPV}. Recall that
a polynomial $p \in \re[x]$ is SOS-convex
(see \cite{HelNie10}) if $\nabla^2 p = Q(x)^TQ(x)$
for some matrix polynomial $Q(x)$.

\begin{theorem} \label{th:charWPV:SOS}
Assume $\mcal{E}=\emptyset$ and $K$ has nonempty interior.
If all $f_i$ and $-c_j$ ($j \in \mcal{I}$) are SOS-convex polynomials,
then the interior of $\mcal{U}$ is equal to
\beq  \label{UPV:SOSCVX}
\mcal{V}_1 \, \coloneqq  \,
\left\{ (v_1,\ldots, v_m)
\left| \baray{l}
\langle c_j, y \rangle \ge 0 \, (j \in \mcal{I}), \\
v_i > \langle f_i, y \rangle  \, (i \in [m]),  \\
 M_{d_0}[y] \succeq 0, \, y_0 = 1,    \\
  y \in \re^{\N^n_{2d_0}}
\earay \right. \right\},
\eeq
where $d_0  \coloneqq  \max\{ \lceil d/2 \rceil,
\lceil \deg(c_j)/2 \rceil (j \in \mcal{I}   \}$.
Moreover, a vector $v \in f(K)$ is a weakly Pareto value
if and only if it lies on the boundary of
$\mbox{cl}\big( \mcal{V}_1 \big)$.
\end{theorem}
\proof
Clearly, if $v$ belongs to $\mcal{V}$ as in \reff{convex:notPV}
for some $y\in \mathscr{R}_d(K)$, then it must belong to $\mcal{V}_1$.
Conversely, if $(v, y)$ satisfies \reff{UPV:SOSCVX},
then let $\hat{x}  \coloneqq  (y_{e_1}, \ldots, y_{e_n})$ and $\hat{y}  \coloneqq  [\hat{x} ]_d$.
Under the SOS convexity assumption,
the Jensen's inequality (see \cite{Las09}) implies that
\[
\langle f_i, y \rangle \ge f_i(\hat{x} ) = \langle f_i, \hat{y} \rangle, \quad
0 \ge  \langle -c_j, y \rangle \ge -c_j(\hat{x} ) = \langle -c_j, \hat{y} \rangle .
%\langle c_j, \hat{y} \rangle =
%c_j(\hat{x} ) \ge \langle c_j, y \rangle \ge 0.
\]
So we have $\hat{x} \in K$ and $\hat{y} \in \mathscr{R}_d(K)$,
hence $v$ belongs to $\mcal{V}_1$.
The conclusion then follows from Theorem~\ref{th:intU:WPV}.
% \Halmos
 \endproof

\begin{example}\label{ex:soscvx boundary}
Consider the SOS-convex polynomials
\[
\baray{c}
f_1 =  (x_1-x_2)^4+(x_2-x_3)^4, \,\,
f_2 = \sum\limits_{i=1}^3x_i^4 + x_1^2x_2^2+x_1^2x_3^2+x_2^2x_3^2,
\earay
\]
and the ball constraint $1 \ge \|x\|^2 $. One can verify that
\[
\baray{l}
    \nabla^2 f_1 =12 \begin{bmatrix}
      x_1-x_2 &  0 \\
      x_2-x_1 & x_2-x_3\\
      0 &  x_3-x_2
    \end{bmatrix}\begin{bmatrix}
      x_1-x_2 &  0 \\
      x_2-x_1 & x_2-x_3\\
      0 &  x_3-x_2
    \end{bmatrix}^T,  \\
    \nabla^2 f_2 = 4\begin{bmatrix}
      x_1 & x_1 & 0 \\
      x_2 & 0 & x_2 \\
      0 & x_3 & x_3
    \end{bmatrix}
    \begin{bmatrix}
      x_1 & x_1 & 0 \\
      x_2 & 0 & x_2 \\
      0 & x_3 & x_3
    \end{bmatrix}^T + \sum\limits_{i=1}^3 A_iA_i^T,
\earay
\]
where each $A_i$ is the diagonal matrix with the diagonal vector
$\sqrt{2}x_i(e+(\sqrt{2}-1)e_i)$. Note that $y_{000}=1$.
The inequalities in the set $\mcal{V}_1$ as in \reff{UPV:SOSCVX} are
\[
\baray{l}
1 - y_{200} - y_{020} - y_{020} \ge 0, \,\,
%%v_1 > \langle f_1, y \rangle , \, v_2 > \langle f_2, y \rangle,
 \\
%%\langle f_1, y \rangle
v_1 > \sum\limits_{i=0}^4 {4 \choose i} (-1)^i(y_{(4-i)e_1+ie_2}+y_{(4-i)e_2+ie_3}),\\
%%\langle f_2, y \rangle
v_2 > \sum\limits_{i=1}^3   y_{4e_i}+ y_{220} + y_{022}+y_{202}. \\
%%\langle f_3, y \rangle =
%%v_3 > y_{400}-4y_{310}+6y_{220}-4y_{130}+y_{040}+y_{004}-y_{100}-y_{010}-y_{001}.
\earay
\]
The moment matrix inequality $M_2[y]\succeq 0$ reads as
{
\[
\left[
\begin{array}{cccccccccc}
y_{000} & y_{100}  & y_{010} & y_{001} & y_{200} & y_{110} & y_{101} & y_{020} & y_{011} & y_{002} \\
y_{100} & y_{200}  & y_{110} & y_{101} & y_{300} & y_{210} & y_{201} & y_{120} & y_{111} & y_{102} \\
y_{010} & y_{110}  & y_{020} & y_{011} & y_{210} & y_{120} & y_{111} & y_{030} & y_{021} & y_{012} \\
y_{001} & y_{101}  & y_{011} & y_{002} & y_{201} & y_{111} & y_{102} & y_{021} & y_{012} & y_{003} \\
y_{200} & y_{300}  & y_{210} & y_{201} & y_{400} & y_{310} & y_{301} & y_{220} & y_{211} & y_{202} \\
y_{110} & y_{210}  & y_{120} & y_{111} & y_{310} & y_{220} & y_{211} & y_{130} & y_{121} & y_{112} \\
y_{101} & y_{201}  & y_{111} & y_{102} & y_{301} & y_{211} & y_{202} & y_{121} & y_{112} & y_{103} \\
y_{020} & y_{120}  & y_{030} & y_{021} & y_{220} & y_{130} & y_{121} & y_{040} & y_{031} & y_{022} \\
y_{011} & y_{111}  & y_{021} & y_{012} & y_{211} & y_{121} & y_{112} & y_{031} & y_{022} & y_{013} \\
y_{002} & y_{102}  & y_{012} & y_{003} & y_{202} & y_{112} & y_{103} & y_{022} & y_{013} & y_{004} \\
\end{array}
\right] \succeq 0.
\]}
%The set $\mcal{V}_1$ as in \reff{UPV:SOSCVX}
%are given by the above inequalities.
\end{example}

%%
%%\begin{figure}
%%\includegraphics[width = 0.6\textwidth]{soscvx_Ex2_gray.png}
%%\caption{Set $\mathcal{U}$ of Example \ref{ex:soscvx boundary}}
%%\label{fig:soscvx boundary}
%%\end{figure}
%%

{
We would like to remark that
the Pareto front can be expressed as an image set of polynomial functions.
Thus, semidefinite relaxations can be used to approximate the Pareto front.
We refer to \cite{MDJ14,MDJ15} for related work on this technique.
In contrast, our work expresses the Pareto front
in terms of the boundary of sets $cl(\mcal{V})$ in \reff{convex:notPV}
or $cl(\mcal{V}_1)$ in \reff{UPV:SOSCVX}.
In comparison, the expression for the Pareto front via
$cl(\mcal{V})$ or $cl(\mcal{V}_1)$ in our work is exact but more for theoretical interest,
while the expression in \cite{MDJ14} is approximate
but more for computational interest.
}

\section{The linear scalarization}
\label{sc:LSP}

This section discusses how to solve linear scalarization problems,
how to choose proper weights, and how to detect nonexistence of proper weights.
For a weight $w  \coloneqq  (w_1, \ldots, w_m)$, denote the weighted sum
\[
f_w(x)  \coloneqq  w_1 f_1(x)+ \cdots +w_m f_m(x).
\]
We consider the LSP
\beq \label{prob:LSP}
\min \quad  f_w(x)  \quad \st \quad x \in K.
\eeq
Recall that $w \ne 0$ is a proper weight if \reff{prob:LSP} is bounded below.
Equivalently, $w$ is a proper weight if and only if
$w$ is the normal of a supporting hyperplane for the set
$\mcal{U}$ as in \reff{set:U}.

\subsection{Tight relaxations for LSPs}
\label{ssc:momLSP}

The Moment-SOS hierarchy of semidefinite relaxations \cite{Las01}
can be applied to solve \reff{prob:LSP}.
When the feasible set $K$ is unbounded, the Moment-SOS hierarchy may not converge.
Here, we apply the tight relaxation method in \cite{Tight19} to solve \reff{prob:LSP}.

The Karush-Kuhn-Tucker (KKT) conditions for \reff{prob:LSP} are
\[
%%\left\{ \begin{array}{c}
\nabla f_w(u)  = \sum\limits_{ i \in \mcal{E} \cup \mcal{I} } \lambda_i \nabla c_i(u), \quad
   \lmd_j \ge 0,\, \lmd_j c_j(u) = 0 \, (j \in \mcal{I} ),
%%\end{array} \right.
\]
where the $\lmd_j$'s are Lagrange multipliers.
For convenience, we write such that
\[
\baray{c}
\mcal{E} \cup \mcal{I} \,=\,  \{ 1, \ldots, s\}, \quad
c \,  \coloneqq  \, (c_1(x), \ldots, c_s(x)), \\
c_{eq} \,  \coloneqq  \, (c_i)_{ i \in \mcal{E} },  \quad
c_{in} \,  \coloneqq  \, (c_j)_{ j \in \mcal{I} }.
\earay
\]
The KKT conditions imply that
\beq \label{mat:C(x)}
\underbrace{\bbm
\nabla c_1(x) & \nabla c_2(x) & \cdots & \nabla c_s(x) \\
c_1(x) & 0  & \cdots & 0 \\
 0     & c_2(x) & \cdots   & 0 \\
\vdots &  \vdots & \ddots & \vdots \\
 0     &  0  & \cdots &  c_s(x) \\
\ebm}_{C(x)}
\underbrace{\bbm \lmd_1 \\ \vdots \\ \lmd_s \ebm}_{\lmd}
=
\bbm \nabla f_w(x) \\ 0 \\ \vdots \\ 0 \ebm.
\eeq
The polynomial tuple $c$ is said to be {\it nonsingular}
if the matrix $C(x)$ as above has full column rank for all complex
$x \in \cpx^n$ (see \cite{Tight19}).
When $c$ is nonsingular, there exists a matrix polynomial $L(x)$
such that $L(x) C(x) = I_s$. Then
\[
\lmd = L(x) \bbm \nabla f_w(x) \\ 0 \ebm .
\]
For each $i=1,\ldots,s$, let
$\lmd_i(x)   \coloneqq   \big( L(x)_{:, 1:n} \nabla f_w(x)\big)_i$
be the $i$th entry polynomial. Denote the polynomial sets
\begin{align}
\label{set:Phi}
&\Phi  \coloneqq   \big \{ c_i \big\}_{ i \in \mcal{E} } \cup
 \big \{\lmd_j(x) c_j \big \}_{ j \in \mcal{I} }
 \cup \big\{ \nabla f_w - \sum\limits_{i\in \mcal{E} \cup \mcal{I} } \lambda_i(x) c_i \big\}, \\
 \label{set:Psi}
&\Psi  \coloneqq  \{c_j, \, \lmd_j(x)  \}_{ j \in \mcal{I} } .
\end{align}
(If $p$ is a vector of polynomials,
then $\{ p \}$ denotes the set of entries of $p$.)
If its minimum value is achieved at a KKT point,
then \reff{prob:LSP} is equivalent to
\beq  \label{prob:LSP KKT}
\left\{ \begin{array}{rl}
\min  & f_w(x) \\
\st  &  p(x) = 0 \, (p \in \Phi), \\
     &  q(x) \ge 0 \, (q \in \Psi). \\
\end{array} \right.
\eeq
Let
$k_0  \coloneqq  \max \{ \lceil \deg(p)/2 \rceil: p \in \Phi \cup \Psi \}$.
For an integer $k \ge k_0$, the $k$th order moment relaxation is
\begin{equation} \label{LSP:mom:kth}
\left\{ \begin{array}{rl}
\min &  \langle f_w,y \rangle \\
\st   &  L_{p}^{(k)}[y] = 0 \, (p \in \Phi), \\
      & L_{q}^{(k)}[y] \succeq 0 \, (q \in \Psi), \\
      &  M_k[y] \succeq 0, \\
      &  y_0 = 1, \, y \in \re^{ \N^n_{2k} }.
\end{array} \right.
\end{equation}
For $k =k_0, k_0+1, \ldots$,
the relaxation \reff{LSP:mom:kth} is a semidefinite program.
The following is the algorithm for solving \reff{prob:LSP KKT}.

\begin{alg} \label{alg:LSP}
  % \rm
Formulate the sets $\Phi, \Psi$ as in \reff{set:Phi}-\reff{set:Psi}.
Let $k  \coloneqq  k_0$.
\bit

\item [Step~1]
Solve the relaxation \reff{LSP:mom:kth}
for a minimizer $y^*$ and let $t \coloneqq k_0$.

\item [Step~2]
If $y^*$ satisfies the rank condition
\beq \label{eq:flatranklow}
 \rank{M_t[y^*]} \,=\, \rank{M_{t-k_0}[y^*]} ,
\eeq
then extract
$r  \coloneqq \rank{M_t[y^*]}$ minimizers for \reff{prob:LSP KKT}.

\item [Step~3]
If \reff{eq:flatranklow} fails to hold and $t < k$,
let $t  \coloneqq  t+1$ and then go to Step~2;
otherwise, let $k  \coloneqq  k+1$ and go to Step~1.

\eit
\end{alg}

The rank condition \reff{eq:flatranklow}
is called {\it flat truncation}.
It is a sufficient (and almost necessary) condition for checking convergence
of the Moment-SOS hierarchy \cite{nie2013certifying}. The Algorithm~\ref{alg:LSP}
can be implemented in {\tt GloptiPoly~3}~\cite{GloPol3}.
The following is the convergence property for the hierarchy of
relaxations \reff{LSP:mom:kth},
which follows from \cite[Theorem~4.4]{Nie2018saddle}.

\begin{theorem}  \label{thmcvg:momLSP}
Assume $c$ is nonsingular and the LSP~\reff{prob:LSP}
has a minimizer for the weight $w$.
Then, for all $k$ large enough,
the optimal value of the relaxation \reff{LSP:mom:kth}
is equal to that of \reff{prob:LSP}.
Moreover, under either one of the following conditions
\bit

\item [(i)]
the set $\ideal[\Phi]+\qmod[\Psi]$ is archimedean, {\bf or}

\item [(ii)]
the real zero set of polynomials in $\Phi$ is finite,

\eit
if each minimizer of \reff{prob:LSP} is an isolated critical point,
then all minimizers of the relaxation \reff{LSP:mom:kth}
must satisfy \reff{eq:flatranklow}, when $k$ is big enough.
Therefore, Algorithm~\ref{alg:LSP} must terminate
within finitely many loops.
\end{theorem}

\begin{example} \label{ex: LSP}
Consider the objectives
\[
\baray{l}
f_1 = \sum\limits_{i=1}^5x_i^4+x_1^2x_2+x_1x_2^2-3x_1x_2x_3+x_3x_4x_5+x_3^3, \\
f_2 = \sum\limits_{i=1}^5x_i^2 -x_1x_2^2-x_2x_3^2+x_3x_4^2+x_4x_5^2
\earay
\]
and the constraint $x_1^2 +\cdots+x_5^2 \ge 1$.
The feasible set is unbounded.
A list of some weights and the corresponding Pareto points
are given in Table~\ref{tab:PP:exm5.3}.\footnote{
Throughout the paper, all computational results are displayed with four decimal digits.
}
\begin{table}[htb]
\begin{center}
\caption{Some Pareto points for Example~\ref{ex: LSP}.}
\label{tab:PP:exm5.3}
\begin{tabular}{|c|c|} \hline
weight $w$ & Pareto point
\\ \hline
$(0.5,0.5)$ & $(-0.3371,    0.4659,  -0.7504,   -0.2807,   -0.1655)$ \\
$(0.25,0.75)$ & $(-0.0986,    0.3316,   -0.6802,   -0.5493,   -0.3405)$ \\
$(0.75,0.25)$ & $(-0.7711,   0.9015,   -1.1818,   -0.5752,   -0.5114)$ \\ \hline
\end{tabular}
\end{center}
\end{table}
\end{example}

It is worthy to note that
\[
\ideal[c_{eq}] \subseteq \ideal[\Phi], \quad
\qmod[c_{in}] \subseteq \qmod[\Psi].
\]
Hence, if $\ideal[c_{in}] + \qmod[c_{in}]$
is archimedean, then the condition (i) in Theorem~\ref{thmcvg:momLSP}
holds. Therefore, if the archimdeanness is met for the constraints
in \reff{prob:mop}, then the condition (i) must hold.

It is possible that $f_w(x)$ is unbounded below on $K$ for some weight $w$.
For instance, $f_w(x)$ is unbounded below for $w=(0,1)$
in Example~\ref{ex: LSP}. We refer to Section~\ref{sc:unbounded}
for how to detect unboundedness. Moreover, we remark that
not every Pareto point is the minimizer of a LSP,
as shown in the following.

\begin{example} \label{ex: CSP}
Consider the MOP with
\[
f_1 \,=\, -x_1^3-x_2^3+(x_3-x_4)^2, \quad
f_2 \,=\, x_1^2-x_2^2+(x_3+x_4)^2
\]
and the constraints $0\le x_1,x_2\le 1$. The LSP is
\[
\left\{ \baray{rl}
\min & w_1 f_1(x) + w_2 f_2(x) \\
  \st &  0\le x_1,x_2\le 1.
\earay \right.
\]
For $w_1\ge w_2$, the minimizer is $(1,1,0,0)$.
For $w_1 < w_2$, the minimizer is $(0,1,0,0)$.
So the LSP can only give two Pareto points,
by exploring all possibilities of weights.
However, each $(x_1,1,0,0)$, with $0\le x_1 \le 1$, is a Pareto point.
%%
%%In contrast, every Pareto point can be found
%%by solving the CSP \reff{CSP:min:x0:AC} for certain weights.
%%We explore all possibilities of weights $w=(w_1,w_2)$
%%by enumerating $0\le w_1\le 1,w_2=1-w_1$ and then get the set of
%%weakly Pareto points from the CSP \reff{CSP:min:x0:AC}.
%%Each obtained weakly Pareto point is also a Pareto point,
%%verified by solving the optimization \reff{detect:Pareto:u}.
%%
\end{example}

\subsection{Existence and choices of proper weights}
\label{ssc:weights}

When $K$ is compact, the LSP \reff{prob:LSP}
is bounded below for all weights. When $K$ is unbounded,
\reff{prob:LSP} may be unbounded below for some $w$ and has no minimizers.
To find a (weakly) Pareto point, we look for a nonzero weight
$w \ge 0$ such that \reff{prob:LSP} is bounded below,
i.e., $w$ is a proper weight.
The set of all proper weights is denoted as
\beq \label{set:W:propW}
\mcal{W} \, \coloneqq  \, \big\{ 0 \ne w \in \re_+^m :
f_w(x) \, \mbox{is bounded below on} \, K \big\}.
\eeq
Clearly, the proper weight set $\mcal{W}$ is a convex cone.

Note that a nonzero weight $w \in \mcal{W}$ if and only if
there exists a scalar $\gamma \in \re$ such that
$f_w(x) - \gamma  \in  \mathscr{P}_d(K)$. So,
\beq \label{W=Pd(K)+R}
\mcal{W} \,= \, \big\{ 0 \ne w  \in \re_+^m :
 f_w(x) \in \mathscr{P}_d(K) + \re  \big\}.
\eeq
The cone $\mathscr{P}_d(K)$
can be approximated by the sum of the ideal $\ideal[c_{eq}]$
and the quadratic module $\qmod[c_{in}]$.
Thus, we have the following.

\begin{proposition}
It holds that
\beq \label{W>=IQ(K)+R}
\big\{ 0 \ne  w  \in \re_+^m :
f_w(x) \in \ideal[c_{eq}] + \qmod[c_{in}] + \re  \big\}
\, \subseteq \, \mcal{W} .
\eeq
\end{proposition}

When $\ideal[c_{eq}] + \qmod[c_{in}]$ is archimedean
($K$ is bounded for this case),
the containment in \reff{W>=IQ(K)+R} is an equality.
This is because if $f_w(x)$ is bounded below on $K$,
then $f_w(x) - \gamma \in \ideal[c_{eq}] + \qmod[c_{in}]$
for $\gamma$ small enough.
When $K$ is unbounded, the sum
$\ideal[c_{eq}] + \qmod[c_{in}]$ cannot be archimedean,
and the containment in \reff{W>=IQ(K)+R} is typically not an equality.
For instance, for $K = \re^3$,
$f_1 = x_1^2x_2^2(x_1^2+x_2^2)$, $f_2=x_3^6-3x_1^2x_2^2x_3^2$,
we have $(1,1) \in \mcal{W}$ but $f_{(1,1)} \not\in \Sig[x]+\re$.
For this case, $\ideal[c_{eq}] = \{0\}$, $\qmod[c_{in}]= \Sig[x]$,
and $f_{(1,1)}$ is the Motzkin polynomial that is
nonnegative but not SOS.

Among all proper weights $w \ge 0$ normalized as $e^Tw =1$,
the smallest possibility of
the minimum value of \reff{prob:LSP} is equal to the smallest one of
$f_1^*, \ldots, f_m^*$,
where $f_i^*$ is the minimum value of $f_i(x)$ on $K$.
Some of $f_i^*$ may be $-\infty$.
For the choice $w = e_i$, the minimum value of \reff{prob:LSP}
is $f_i^*$. Beyond them, people are also interested in
$w$ such that the minimum value of \reff{prob:LSP} is maximum.
We discuss how to find such $w$ in the following.

Assume $d$ is the maximum degree of $f_1,\ldots, f_m$.
%\begin{equation}  \label{degree:d}
%d \,  \coloneqq  \, \max \big\{
%\{  \deg(f_i) \}_{i \in [m]},\,
%\{  \deg(c_i) \}_{i \in \mcal{E} \cup \mcal{I} }
%\big\}.
%\end{equation}
For the minimum value of $f_w(x)$ on $K$ to be maximum,
we consider the optimization
\beq \label{maxgm:fw-gm>=0}
\left\{ \baray{rl}
\max & \gamma \\
\st & 1 - e^T w = 0, \,  w_1 \ge 0, \ldots, w_m \ge 0,  \\
    & \sum\limits_{i=1}^m w_i f_i - \gamma \, \in \, \mathscr{P}_d(K).
\earay \right.
\eeq
The dual cone of $\mathscr{P}_d(K)$ is $\cl{ \mathscr{R}_d(K) }$.
(When $K$ is compact, the moment cone $\mathscr{R}_d(K)$ is closed.)
%%%%%%%%%%%%%%%%%%%%%%%%%%%%%%%%%%%%%%%%%%%%%%%%%%%%%%%
\iffalse

The Lagrange function for \reff{maxgm:fw-gm>=0} is
\[
\baray{rcl}
\mcal{L}(x, \lmd, \mu, y) & \coloneqq & \gamma + \mu (1 - e^T w ) +
\sum\limits_{i=1}^m \lmd_i  w_i + \langle  \sum_i w_i f_i(x) - \gamma, y \rangle  \\
 &=& \mu + \gamma (1-y_0) + \sum\limits_{i=1}^m
 w_i (-\mu + \lmd_i  + \langle f_i , y \rangle ),
\earay
\]
where $\lmd  \coloneqq (\lmd_1, \ldots,\lmd_m) \ge 0$, $\mu \in \re$,
and $y \in \cl{ \mathscr{R}_d(K) }$.

\fi
%%%%%%%%%%%%%%%%%%%%%%%%%%%%%%%%%%%%%%%%%%%%%%%%%%%%%%%%%%%%%%%
The dual optimization of \reff{maxgm:fw-gm>=0} is
\beq \label{min:mu>=fiy:RdK}
\left\{ \baray{rl}
\min &  \mu  \\
\st & \mu - \langle f_i, y \rangle \ge 0 \,\, (i=1,\ldots, m), \\
    & y_0 = 1, \, y \in \cl{ \mathscr{R}_d(K) }.
\earay \right.
\eeq
The $k$th order SOS relaxation for \reff{maxgm:fw-gm>=0} is
\beq \label{maxgm:fw:IQ2k}
\left\{ \begin{array}{rl}
\max  & \gamma \\
\st & w_1+\cdots w_m = 1, \, w_1 \ge 0, \ldots, w_m \ge 0,  \\
   &  \sum\limits_{i=1}^m w_i f_i - \gamma \, \in \,
   \ideal[c_{eq}]_{2k} + \qmod[c_{in}]_{2k} .
\end{array} \right.
\eeq
The dual optimization of \reff{maxgm:fw:IQ2k}
is the $k$th order moment relaxation for \reff{min:mu>=fiy:RdK}:
\beq \label{min:mu>=fiy:momkth}
\left\{ \baray{rl}
\min &  \mu  \\
\st & \mu - \langle f_i, y \rangle \ge 0 \, (i = 1, \ldots, m),  \\
    &  L_{c_i}^{(k)}[y] = 0 \, (i \in \mcal{E}), \\
    &   L_{c_j}^{(k)}[y] \succeq 0 \, (j \in \mcal{I}), \\
    &  M_k[y]\succeq 0, \\
    & y_0 = 1, \, y \in \re^{ \N^n_{2k} }.
\earay \right.
\eeq
As $k$ increases, the above gives a hierarchy of Moment-SOS relaxations
for solving \reff{maxgm:fw-gm>=0}.
When the sum $\ideal[c_{eq}] + \qmod[c_{in}]$ is archimedean,
the convergence of the hierarchy was shown in \cite{Las08,linmomopt}.

\begin{example}
Consider the objectives
\[
\baray{l}
f_1 = {\left({x_{1}}^2+x_{2}+x_{3}\right)}^2+{\left({x_{2}}^2+x_{3}
   +x_{4}\right)}^2-3\,x_{1}x_{2}x_{3}x_{4},  \\
f_2=\sum\limits_{i=1}^4x_i^4-\left(x_{1}-x_{2}\right)\left(x_{2}-x_{3}\right)
    \left(x_{3}-x_{4}\right)\left(x_{4}-x_{1}\right),  \\
f_3 = 3\sum\limits_{i=1}^4x_i^3 +x_{1}^2\left({x_{2}}^2-{x_{3}}^2\right)
 +x_{2}^2\left({x_{3}}^2-{x_{4}}^2\right) +x_{3}^2\left({x_{4}}^2-{x_{1}}^2\right)
\earay
\]
and the constraints $x_1x_2 \ge 1, x_2x_3 \ge 1, x_3x_4 \ge 1, x_1\ge 0$.
Each $f_i$ is unbounded below on the feasible set $K$.
The optimization \reff{maxgm:fw-gm>=0} can be solved by the
Moment-SOS hierarchy of \reff{maxgm:fw:IQ2k}-\reff{min:mu>=fiy:momkth}.
The computed optimal weight $w^*$
and Pareto point $x^*$ are respectively
\[
w^* = (0.5769, 0.2229, 0.2003), \,
x^* = ( 1.0105, 0.9897, 1.0105, 0.9897).
\]
The maximum of the minimum value of $f_w(x)$ on $K$ is
$\gamma^* = 11.9435$.
\end{example}

\subsection{Nonexistence of proper weights}
\label{ssc:noprop:weight}

When the feasible set $K$ is unbounded,
there may not exist a weight $w \ge 0$ such that
$f_w(x)$ is bounded below on $K$.
We discuss how to detect nonexistence of proper weights.

Recall that $d$ is the maximum degree of $f_i$ and
$\widetilde{f_w} (\tx)  \coloneqq   x_0^{d} f_w( \frac{x}{x_0} )$.
When $K$ is closed at $\infty$, the optimization
\reff{maxgm:fw-gm>=0} is equivalent to
\begin{equation} \label{prob:LSP:wihmg}
\left\{ \begin{array}{rl}
\max  &  \gamma \\
\st  & w_1+\cdots w_m = 1, \, (w_1, \ldots, w_m) \ge 0, \\
     & \widetilde{f_w} - \gamma x_0^{d} \in
     \mathscr{P}_d( \widetilde{K} ) .
\end{array} \right.
\end{equation}
The dual optimization of \reff{prob:LSP:wihmg} is
\beq \label{minmu:ty:RdKhmg}
\left\{ \baray{rl}
\min &  \mu  \\
\st & \mu - \langle x_0^{d} f_i(x/x_0), \ty \rangle \ge 0 \, (i =1,\ldots, m), \\
    & \langle x_0^{d}, \ty \rangle = 1, \,
    \ty \in \mathscr{R}_d( \widetilde{K} ).
\earay \right.
\eeq
When \reff{minmu:ty:RdKhmg} is unbounded below,
the problem \reff{prob:LSP:wihmg} must be infeasible,
and hence there is no proper weight.
This is the case if \reff{minmu:ty:RdKhmg} has a decreasing ray $\Delta \ty$:
\beq \label{-1:Dty:RdKhmg}
\left\{ \baray{l}
 -1 \ge \langle x_0^{d} f_i(x/x_0),\Delta \ty \rangle \,\, (i =1,\ldots, m), \\
\langle x_0^{d} , \Delta \ty \rangle = 0, \,\,
\Delta \ty \in \mathscr{R}_d( \widetilde{K} ).
\earay \right.
\eeq
Let $f_i^{(d)}$ denote the homogeneous part of degree $d$ for $f_i$, i.e.,
\[
f_i^{(d)}  = x_0^{d} f_i(x/x_0)\big|_{x_0 = 0}.
\]
%%(If $d > \deg(f_i)$ then $f_i^{(d)} = 0$.)
The equality $\langle x_0^{d}, \Delta \ty \rangle = 0$
implies that every representing measure for $\Delta \ty$
must be supported in the hyperplane $x_0 = 0$.
Therefore, \reff{-1:Dty:RdKhmg} can be reduced to
\beq \label{-1:Dy:RdK}
 -1 \ge \langle  f_i^{(d)},\Delta y \rangle \,\, (i =1,\ldots, m), \quad
\Delta y \in \mathscr{R}_d(K^\circ),
\eeq
where $K^\circ$ is the set as in \reff{set:K0:unbdPOP}.
%
%\beq
%\widetilde{K}^\circ =
%\left\{ x \in \re^n \left|\baray{l}
% c_i^{hom}(x) = 0 (i \in \mcal{E}), \\
% c_j^{hom}(x) \ge 0 (j \in \mcal{I}), \\
%  x^Tx = 1
% \earay \right. \right\}.
%\eeq
%
We remark that if $\deg(f_i) < d$, then
$f_i^{(d)} = 0$ and hence $\langle f_i^{(d)}, \Delta y \rangle = 0$,
which implies that \reff{-1:Dty:RdKhmg} is infeasible.
Therefore, the decreasing ray $\Delta \ty$ as in \reff{-1:Dty:RdKhmg}
exist only if all $f_i$ have the same degree.
The following is the nonexistence theorem of proper weights.
Like before, the closeness of $K$ at infinity can be weakened.

\begin{theorem} \label{thm:noprop:w}
Assume \reff{-1:Dy:RdK} has a feasible point
$\Delta y = \lmd_1 [z_1]_d + \cdots + \lmd_r [z_r]_d$,
with $\lmd_1, \ldots, \lmd_r >0$ and $z_1,\ldots,z_r \in K^\circ$.
If each $(0,z_i)$ lies on $cl \big( \widetilde{K}\cap \{x_0>0\} \big)$,
then the LSP \reff{prob:LSP} is unbounded below for all nonzero
$w \ge 0$ and hence $\mcal{W} = \emptyset$.
\end{theorem}
\proof
For each $w \ge 0$ with $e^Tw=1$, it holds that
\[
\baray{c}
 -1 \ge \langle  \sum\limits_{i=1}^m w_i f_i^{(d)},\Delta y \rangle
 = \langle  \widetilde{f_w}, \Delta y \rangle ,  \quad
\Delta y \in \mathscr{R}_d(K^\circ) .
\earay
\]
Since $\Delta y = \lmd_1 [z_1]_d + \cdots + \lmd_r [z_r]_d$,
there exists at least one $i$ such that
\[
-1/r \, \ge \,   \langle  \widetilde{f_w}, \lmd_i   [z_i]_d \rangle.
\]
By Theorem~\ref{thm:unbounded:pop}(ii), $f_w(x)$ is unbounded below on $K$,
since $(0,z_i)$ lies in the closure of
$\widetilde{K}\cap \{x_0>0\}$ and $\lmd_i >0$.
A nonzero weight $w \ge 0$ is proper if and only if $w/(e^Tw)$ is proper.
Hence, no proper weights exist and $\mcal{W} = \emptyset$.
% \Halmos
\endproof

The moment system \reff{-1:Dy:RdK} is in the form \reff{-1>=giz:several}.
Algorithm~\ref{alg:ubd} can be applied to
get a feasible point for \reff{-1:Dy:RdK}.
This can be done by solving a hierarchy of
moment relaxations like \reff{min<Ry>:mom}.
%%%%%%%%%%%%%%%%%%%%%%%%%%%%%%%%%%%%%%%%%%%%%%%%%%%%%%%%%%%%%%%%%%%
\iffalse

Denote the degree
\beq \label{halfdegree:d1}
 d_1 \,  \coloneqq  \, \max \Big\{
\big\{ \lceil \deg(f_i)/2 \rceil \big\}_{i \in [m]},
\big\{ \lceil \deg(c_j)/2 \rceil \big\}_{j \in \mcal{E} \cup \mcal{I} }
\Big\}.
\eeq
Choose a generic matrix $R \in \inte{ \Sig[x]_{2d_1} }$ and solve the moment relaxation
\beq  \label{min<Ry>:noprop:w}
\left\{ \begin{array}{rl}
\min  & \langle R, z \rangle \\
\st  &  1+\langle f_i^\hm, z\rangle \le 0, (i \in [m]) \\
     & L^{(k)}_{x^Tx-1}[z] = 0, \, L^{(k)}_{c_i^{hom} }[z] = 0 \, (i \in \mcal{E}),   \\
     & L^{(k)}_{c_j^{hom} }[z] \succeq 0 \, (j \in \mcal{I}),   \\
     & M_k[z] \succeq 0, \, z \in \re^{ \N^n_{2k} },
\end{array} \right.
\eeq
for $k= d_1, d_1+1, \cdots$.
Suppose $z^*$ is a minimizer of \reff{min<Ry>:noprop:w}.
If there exists an integer $t \in [d_1, k]$ such that
\beq \label{rank:FT:now>=0}
\rank \, M_t[z^*] \, = \,
\rank \, M_{t-d_1}[z^*],
\eeq
then the truncation $\Delta y  \coloneqq  z^*|_d$
has a representing measure supported in $\widetilde{K}^\circ$.
%%The rank condition \reff{rank:FT:now>=0}
%%is called the flat truncation \cite{nie2013certifying}.
When it is satisfied, the truncation $z^*|_d$
must be a feasible point of \reff{-1:Dy:RdK}.

\fi
%%%%%%%%%%%%%%%%%%%%%%%%%%%%%%%%%%%%%%%%%%%%%%%%%%%%%%%%%%%%%%%
The convergence is shown in Theorem~\ref{thm:cvg:ubd}.

\begin{example}
Consider the objectives
\[
\baray{l}
f_1 = -\big( \sum\limits_{i=1}^5 x_i^3 \big)
      - x_2^4 + x_4^4-x_1x_2x_3-x_3x_4x_5, \\
f_2 =  \big(\sum\limits_{i=1}^5 x_i \big)^3
    - \sum\limits_{i=1}^4 x_i^4 + x_1x_2x_3x_4 + x_2x_3x_4x_5, \\
f_3 = x_1^4-x_2^4+x_3^4+x_4^4-x_1x_2x_3-x_3x_4x_5, \\
f_4 = -(x_1x_2)^2+ (x_2x_3)^2+ (x_3x_4)^2+ (x_4x_5)^2
\earay
\]
and the constraints $x_1^2\ge 1, \ldots, x_5^2 \ge 1$.
By Algorithm~\ref{alg:ubd}, we get that
$\Dt y = \lmd [u]_4$ is feasible for \reff{-1:Dy:RdK} with
\[
u = ( -0.7014,   -0.7049,    0.0533,   -0.0428,    0.0803),\quad \lmd = 4.1146.
\]
The set $\mathcal{C}$ as in \eqref{eq:direction} is empty.
By Lemma \ref{lemma: cond closure}, the point $(0,u)$ lies on the closure of $\widetilde{K}\cap\{x_0>0\}$.
Therefore, the LSP \reff{prob:LSP} is unbounded below
for all nonzero weights $w \ge 0$ by Theorem~\ref{thm:noprop:w}.
\end{example}

We remark that when no proper weights exist,
the system \reff{-1:Dy:RdK} is still possibly infeasible.
For instance, this is the case for
\[
K = \re^1, \quad  f_1 = x_1^3+x_1, \quad f_2 = -x_1^3.
\]
There is no nonzero $(w_1, w_2)\ge 0$ such that $f_w(x)$
is bounded below on $\re^1$. However, there is no $\Dt y$ such that
\[
-1 \ge \langle x_1^3, \Dt y \rangle, \quad
-1 \ge \langle -x_1^3, \Dt y \rangle, \quad
\Dt y \in \mathscr{R}_3( \{ x_1^2 = 1 \} ).
\]
Moreover, when no proper weights exist, Pareto points may still exist.
For instance, this is the case for
\[
\left\{ \baray{rl}
\min  &  (x_1, x_2) \\
\st   &  x_1 + x_2^3 \geq 0 .
\earay \right.
\]
For every $t$, $(t^3, -t)$ is a Pareto point,
but there is no nonzero $w=(w_1, w_2) \ge 0$ such that
$w_1x_1+w_2x_2$ is bounded below on $x_1+x_2^3 \ge 0$.
%
%The set of Pareto values are shown as
%bold points on the boundary of Figure~\ref{fig: no prop w}.
%\begin{figure}[htb]
%\includegraphics[width = 0.68\textwidth]{NoW_ParetoExist.png}
%\caption{Pareto points exist with no proper weights}
%\label{fig: no prop w}
%\end{figure}
%

\section{The Chebyshev scalarization}
\label{sc:CSP}

The Chebyshev scalarization problem is
\beq  \label{prob:CSP}
 \begin{array}{rl}
 \min\limits_{x \in K}  & \max\limits_{ 1 \le i \le m } w_i(f_i(x)- f_i^*)
\end{array}
\eeq
for a nonzero weight $w  \coloneqq  (w_1, \ldots, w_m) \ge 0$.
In the above, each $ f_i^*$
is the minimum value of $f_i$ on $K$.
In this section, we assume all $f_i^* > -\infty$.
If one of them is $-\infty$, we refer to Subsection~\ref{ssc:exist:wpp}
for how to get PPs and WPPs.

Each minimizer of \reff{prob:CSP} is a weakly Pareto point.
Conversely, every weakly Pareto point is a minimizer of the CSP \reff{prob:CSP}
for some weight, provided each $f_i^* > -\infty$.
This is because if $x^*$ is a weakly Pareto point,
then there exist weights $w_i \ge 0$ such that
all $w_i(f_i(x^*)-f_i^*) $ are equal,
since $f_i(x^*)- f_i^* \ge 0$ for each $i$.
Then $x^*$ is the minimizer for that CSP.
Observe that $f_i^*$ equals the minimum value of the LSP \reff{prob:LSP}
for the weight $w = e_i$.
Algorithm~\ref{alg:LSP} can be applied to compute $f_i^*$.

After all $f_i^*$ are obtained,
one can solve the CSP \reff{prob:CSP} for a weakly Pareto point.
With the new variable $x_{n+1}$, the CSP~\reff{prob:CSP} is equivalent to
\beq \label{CSP:min:x0}
\left\{ \begin{array}{rl}
\min  &  x_{n+1} \\
\st   & x_{n+1} - w_i(f_i(x)- f_i^*) \ge 0 \, (i=1,\ldots,m), \\
      & c_i(x) = 0 \, (i \in \mcal{E}), \\
      & c_j(x) \ge 0 \, (j \in \mcal{I}) .
\end{array} \right.
\eeq
To get convergent Moment-SOS relaxations,
we typically need archimedeanness for constraining polynomials.
The feasible set of \reff{CSP:min:x0} is unbounded.
To fix this issue, one can select a feasible point $\xi \in K$ and let
\[
B_0 \,  \coloneqq  \, \max\limits_{1 \le i \le m} \big( w_i(f_i(\xi)- f_i^*) \big).
\]
Then \reff{CSP:min:x0} is equivalent to
\beq \label{CSP:min:x0:AC}
\left\{ \begin{array}{rl}
\min  &  x_{n+1} \\
\st   & x_{n+1} - w_i(f_i(x)- f_i^*) \ge 0\, (i=1,\ldots,m), \\
      & B_0 - x_{n+1} \ge 0,\,\, x_{n+1} \ge 0, \\
      & c_i(x) = 0 \, (i \in \mcal{E}), \\
      & c_j(x) \ge 0 \, (j \in \mcal{I}) .
\end{array} \right.
\eeq
For convenience, denote the set
\beq \label{set:mcG}
\mcal{G}  \coloneqq  \big\{ c_j \big\}_{ j \in \mcal{I} } \cup
\{x_{n+1},\, B_0 - x_{n+1}\} \cup
\big \{ x_{n+1} - w_i(f_i- f_i^*) \big \}_{i=1}^m .
\eeq
The $k$th order moment relaxation for \reff{CSP:min:x0:AC} is
\beq \label{CSP:momx0:kth}
\left\{ \begin{array}{rl}
\min &  \langle x_{n+1},y \rangle \\
\st   &  L_{c_i}^{(k)}[y] = 0 \, (i \in \mcal{E}), \\
      & L_{p}^{(k)}[y] \succeq 0 \, (p \in \mcal{G}), \\
     &   M_k[y] \succeq 0, \\
     & y_0 = 1, \,  y \in \re^{ \N^{n+1}_{2k} }.
\end{array} \right.
\eeq
Let $d_0$ be the degree
\beq \label{deg:d0:ubd}
d_0 \,  \coloneqq  \, \max \big\{ \lceil d/2 \rceil,
\lceil \deg(c_i)/2 \rceil \, (i \in \mcal{E} \cup \mcal{I})  \big\} .
\eeq
%%
%%\beq \label{halfdegree:d1}
%%d_1 \,  \coloneqq  \, \max \Big\{
%%\big\{ \lceil \deg(f_i)/2 \rceil \big\}_{i \in [m]},
%%\big\{ \lceil \deg(c_j)/2 \rceil \big\}_{j \in \mcal{E} \cup \mcal{I} }
%%\Big\}.
%%\eeq
%%
Suppose $y^*$ is a minimizer of \reff{CSP:momx0:kth}.
If there exists $t \in [d_0, k]$ such that
\beq \label{flat:csp}
\rank\, M_t[y^*] \, = \, \rank\, M_{t-d_0}[y^*] ,
\eeq
then we can get $\rank\, M_t[y^*]$ minimizers for \reff{prob:CSP}
(see \cite{HenLas05,nie2013certifying}).
The following is about the convergence of the hierarchy of
\reff{CSP:momx0:kth}.

\begin{theorem} \label{cvg:CSP:momsos}
Assume $\ideal[c_{eq}]+\qmod[c_{in}]$ is archimedean.
Suppose $y^{(k)}$ is a minimizer of the moment relaxation \reff{CSP:momx0:kth}
for the order $k$. If the CSP \reff{prob:CSP}
has finitely many minimizers, then for $t$ big enough,
every accumulation point of $\{y^{(k)}|_{2t} \}_{k=d_0}^\infty$
must satisfy \reff{flat:csp}.
\end{theorem}
\proof
Since $\ideal[c_{eq}]+\qmod[c_{in}]$ is archimedean,
there exists a scalar $N$ such that
$
N - x^Tx \, \in \, \ideal[c_{eq}]+\qmod[c_{in}].
$
Note that
\[
B_0^2 - x_{n+1}^2 = (B_0 - x_{n+1})^2 + 2x_{n+1}\cdot \frac{(B_0-x_{n+1})^2}{B_0}
+ 2(B_0-x_{n+1}) \frac{x_{n+1}^2}{B_0} .
\]
Therefore, we get that
\[
N - x^Tx + B_0^2 - x_{n+1}^2 \, \in \, \ideal[c_{eq}]+\qmod[\mcal{G}].
\]
This means that $\ideal[c_{eq}]+\qmod[\mcal{G}]$ is archimedean.
When the CSP \reff{prob:CSP} has finitely many minimizers,
the conclusion is implied by Theorem~3.3 of \cite{nie2013certifying}.
% \Halmos
\endproof

When $\ideal[c_{eq}]+\qmod[c_{in}]$ is not archimedean
(this is the case if $K$ is unbounded),
the homogenization method in Subsection~\ref{ssc:weights}
can be similarly applied. Moreover, the method in \cite{MLM19}
can also be applied to solve \reff{prob:CSP}.

\begin{example} \label{exm6.2}
Consider the objectives
\[
\baray{l}
f_1 = \sum\limits_{i=1}^4 x_i^2 -(x_1x_2+x_3x_4)(x_1x_3+x_2x_4), \\
f_2 = \sum\limits_{i=1}^4x_i^4+x_1x_2x_3+x_2x_3x_4 + x_1x_2x_3x_4, \\
f_3 =  \sum\limits_{i=1}^4x_i^6 + (x_1^2-x_2^2+1)(x_2^2-x_3^2+1)(x_3^2-x_4^2+1)
\earay
\]
and the constraint $x_1x_2 \le 1,x_2x_3 \le 1,x_3x_4 \le 1,x_1x_4 \le 1$.
The minimum values $f_1^*, f_2^*, f_3^*$
are $0.0000,-0.0710,0.6029$ respectively.
A list of some weights and corresponding weakly Pareto points
are in Table~\ref{tab:exm6.2}.
Indeed, they are all Pareto points,
confirmed by solving the optimization~\reff{detect:Pareto:u}.
\begin{table}[htb]
\begin{center}
\caption{Some Pareto points for Example~\ref{exm6.2}.}
\label{tab:exm6.2}
\begin{tabular}{|c|c|} \hline
weight $w$ & Pareto point \\ \hline
$(1,1,1)$ & $(0.000,0.000,0.000,0.4503)$ \\
$(1,2,2)$ & $(-0.0024,   -0.0979,   -0.0635,   -0.5248)$ \\
$(1,2,3)$ & $(-0.0029,   -0.1228,   -0.0700,   -0.5648)$ \\ \hline
\end{tabular}
\end{center}
\end{table}
\end{example}

\section{Existence and nonexistence of PPs and WPPs}
\label{sc:detect}

This section discusses how to check if a given point
is a (weakly) Pareto point and
how to detect existence or nonexistence of (weakly) Pareto points.

\subsection{Detection of PPs and WPPs}
\label{ssc:detec:PP}

For a given point $x^* \in K$, how can we detect if it is a Pareto point or not?
To this end, consider the optimization
\beq \label{detect:Pareto:u}
\left\{ \begin{array}{rl}
\min  & f_e(x) \coloneqq  f_1(x)+\cdots+f_m(x) \\
\st   & f_i(x^*) - f_i(x) \ge 0 \, \, (i=1,\ldots,m), \\
      & x \in K.
\end{array} \right.
\eeq
This is a kind of lexicographic method (see \cite{marler2004survey}).
Let $z^*$ be a minimizer of \reff{detect:Pareto:u}, if it exists.
Then, $x^*$ is a Pareto point if and only if
the minimum value of \reff{detect:Pareto:u} is equal to $f_e(x^*)$.
Moreover, if $x^*$ is not a Pareto point,
the minimizer $z^*$ must be a Pareto point,
since all the weights are positive.
A Pareto point may be obtained by solving \reff{detect:Pareto:u}
for given $x^* \in K$, provided \reff{detect:Pareto:u} has a minimizer.

Let $F$ be the feasible set of \reff{detect:Pareto:u} and
\beq \label{polytuple:F}
\mcal{F} \, \coloneqq \, \big\{ c_j \big\}_{ j \in \mcal{I} } \cup
 \big \{ f_i(x^*) - f_i(x)) \big \}_{i=1}^m .
\eeq
For a degree $k \ge d/2$, the $k$th order moment relaxation
for \reff{detect:Pareto:u} is
\beq \label{detect:wP:momkth}
\left\{ \begin{array}{rl}
\min &  \langle f_e,y \rangle \\
\st  &  L_{c_i}^{(k)}[y] = 0 \, (i \in \mcal{E}), \\
     &  L_{q}^{(k)}[y] \succeq 0 \, (q \in \mcal{F}), \\
     &   M_k[y] \succeq 0, \\
     & y_0 = 1, \,  y \in \re^{ \N^{n}_{2k} }.
\end{array} \right.
\eeq
Recall that $d_0$ is the degree as in \reff{deg:d0:ubd}.
%Let $d_1$ be the degree
%\[
%d_1  \coloneqq  \max\Big\{ \lceil \deg(p)/2 \rceil : \,
% p \in C_\mcal{E} \cup \mcal{F} \} \Big \} .
%\]
Suppose $y^*$ is a minimizer of \reff{detect:wP:momkth}.
If there exists $t \in [d_0, k]$ such that
\beq \label{flat:detcP}
\rank\, M_t[y^*] \, = \, \rank\, M_{t-d_0}[y^*] ,
\eeq
then we can get $r \coloneqq  \rank\, M_t[y^*]$ minimizers for \reff{detect:Pareto:u}.
Recall that $c_{eq}$ is the tuple of equality constraining polynomials.
The following result follows from Theorem~3.3 of \cite{nie2013certifying}.

\begin{theorem}
\label{cvg:detct:PP}
Assume $\ideal[c_{eq}]+\qmod[\mcal{F}]$ is archimedean.
Suppose $y^{(k)}$ is a minimizer of the relaxation \reff{detect:wP:momkth}
for the order $k$. If \reff{detect:Pareto:u}
has only finitely many minimizers, then for $t$ big enough,
every accumulation point of $\{y^{(k)}|_{2t} \}_{k=1}^\infty$
must satisfy \reff{flat:detcP}.
\end{theorem}

When $\ideal[c_{eq}]+\qmod[\mcal{F}]$ is not archimedean,
the hierarchy of relaxations \reff{detect:wP:momkth}
may not converge. For such a case, we refer to
the homogenization method in Subsection~\ref{ssc:weights}
or the method in \cite{MLM19}.

\begin{example}
(i) Consider the objectives
\[
\baray{l}
f_1 = x_1^2(x_1-2)^2 + (x_1-x_2)^2+(x_2-x_3)^2+(x_3-x_4)^2, \\
f_2 = -x_1^2-x_2^2-x_3^2-x_4^2+x_1x_2+x_2x_3+x_3x_4
\earay
\]
and the constraint $x \ge 0$.
We first solve the CSP~\reff{prob:CSP}
with $w_1=w_2=1$ and %%relaxation order $k=4$
get the weakly Pareto point $x^*=(0,0,0,0)$.
%%{The point value is
%%around $10^{-5}$. Should we specify this?}.
It is not a Pareto point. By solving \reff{detect:Pareto:u},
%% with relaxation order $k=4$,
we get the Pareto point $(2.000,2.001,2.001,2.001)$. \\
(ii) Consider the objectives
\[
f_1 = x_1^3-x_1^2x_2-x_2, \quad  f_2 = x_2^3-x_1x_2^2-x_1
\]
and the constraint $x_1x_2\le 1$.
The LSP~\reff{prob:LSP} is unbounded below for all weights $w_i$,
which is confirmed by a feasible point for \reff{-1:Dy:RdK}.
But we are still able to find a Pareto point by solving
\reff{detect:Pareto:u} for some given $x^*$.
For instance, for $x^*=(-1,-0.5)$, solving~\reff{detect:Pareto:u}
%%  relaxation order $k=4$
gives the Pareto point $(1.0000,1.0000)$.
\end{example}

%
%
%\subsection{Detection of weakly Pareto Points}
%\label{ssc:detec:wPP}
%

We can similarly detect if a given point $x^* \in K$
is a weakly Pareto point or not. Consider the optimization
\beq \label{detect:Pareto:weak}
\left\{ \begin{array}{rl}
\min  & \max\limits_{1 \le i \le m} \big( f_i(x) - f_i(x^*) \big) \\
\st   & f_i(x^*) - f_i(x) \ge 0 \, (i=1,\ldots, m), \\
      & c_i(x) = 0 \, (i \in \mcal{E}), \\
      &  c_j(x) \ge 0 \, (j \in \mcal{I}) .
\end{array} \right.
\eeq
Let $z^*$ be a minimizer of \reff{detect:Pareto:weak}, if it exists.
Then, $x^*$ is a weakly Pareto point if and only if
the optimal value of \reff{detect:Pareto:weak} is equal to $0$.
Moreover, if $x^*$ is not a weakly Pareto point,
then one can show that $z^*$ is a weakly Pareto point.
%%%%%%%%%%%%%%%%%%%
\iffalse

This is because if there is $x \in K$ such that
$(f_1(x), \ldots, f_m(x)) < (f_1(z^*), \ldots, f_m(z^*))$, then
\[
\max\limits_{1 \le i \le m} \big( f_i(x) - f_i(x^*) \big) \, < \,
\max\limits_{1 \le i \le m} \big( f_i(z^*) - f_i(x^*) \big).
\]

\fi
%%%%%%%%%%%%%%%%%%%%%%%%%%%%%%%%%%
By introducing the new variable $x_{n+1}$, the optimization
\reff{detect:Pareto:weak} is equivalent to
\beq  \label{detect:Pareto:weakx0}
\left\{ \begin{array}{rl}
\min  & x_{n+1}  \\
\st   & x_{n+1} - f_i(x) + f_i(x^*) \ge 0\, (i=1,\ldots, m),  \\
      & f_i(x^*) - f_i(x) \ge 0 \, (i=1,\ldots,m), \\
      & c_i(x) = 0 \, (i \in \mcal{E}), \\
      & c_j(x) \ge 0 \, (j \in \mcal{I}) .
\end{array} \right.
\eeq
The optimal value of \reff{detect:Pareto:weakx0}
is always less than or equal to $0$.
A similar hierarchy of moment relaxations like \reff{CSP:momx0:kth}
can be applied to solve \reff{detect:Pareto:weakx0},
and a similar convergence result like Theorem~\ref{cvg:CSP:momsos} holds.
When the feasible set of \reff{detect:Pareto:weak} is unbounded,
the Moment-SOS hierarchy may not converge.
For such a case, we refer to
the homogenization method in Subsection~\ref{ssc:weights}
or the method in \cite{MLM19}.

\subsection{Existence of PPs and WPPs}
\label{ssc:exist:wpp}

When $K$ is unbounded, we discuss how to detect existence of PPs and WPPs.
Consider the min-max optimization
\beq \label{min:max(f):K}
\min\limits_{x \in K} \,   \max\limits_{ 1 \le i \le m } f_i(x)  .
\eeq
The following is the existence result.
See Subsection~\ref{ssc:support} for $\pi$-minimal points.

\begin{theorem} \label{thm: exist of Pareto}
The min-max optimization \reff{min:max(f):K} has the following properties:

\bit

\item [(i)]
If \reff{min:max(f):K} is unbounded below,
then there is no weakly Pareto point, and hence there is no Pareto point.
If \reff{min:max(f):K} is bounded below,
then every minimizer of \reff{min:max(f):K}
(if it exists) is a weakly Pareto point.

\item [(ii)]
Let $S$ be the set of minimizers of \reff{min:max(f):K}.
For each $x^* \in S$, if $f(x^*)$ is a $\pi$-minimal point of the image $f(S)$
for a permutation $\pi$ of $(1,\ldots,m)$,
then $x^*$ is a Pareto point. In particular, if $S$ is compact,
then there exists a Pareto point.

\eit
\end{theorem}
\proof
(i) If \reff{min:max(f):K} is unbounded below,
then for every $x\in  K$, there exists $z \in K$ such that
\[
\max_{ 1 \le i \le m } f_i(z) \, < \, \min_{ 1 \le i \le m } f_i(x).
\]
This implies $f(z) < f(x)$, hence there is no weakly Pareto point.

Suppose \reff{min:max(f):K} is bounded below and it has a minimizer,
say, $x^*$. Then $x^*$ must be a weakly Pareto pint.
If otherwise there is $z \in  K$ such that $f(z) < f(x^*)$, then
\[
 \max_i f_i(z) \,<\, \max_i f_i(x^*),
\]
which contradicts that $x^* $ is a minimizer.

(ii) Suppose $f(x^*)$ is a $\pi$-minimal point of $f(S)$.
Let $z \in  K$ be a point such that $f(z) \le f(x^*)$.
Since $x^*$ is a minimizer of \reff{min:max(f):K}, one can see that
\[
\max_{1 \le i \le m } f_i(x^*) \, \le \,
\max_{1 \le i \le m } f_i(z) \, \le \,
\max_{1 \le i \le m } f_i(x^*).
\]
This implies that $z$ is also a minimizer of \reff{min:max(f):K},
so $z \in S$. Since $f(x^*)$ is $\pi$-minimal among $f(S)$,
$f(x^*) \le f(z)$, so $f(x^*) = f(z)$
and hence $x^*$ is a Pareto point.
When $S$ is compact, the set $S$ must have
a $\pi$-minimal point, for every permutation $\pi$ of $(1,\ldots,m)$,
and hence \reff{prob:mop} has a Pareto point,
by Proposition~\ref{pro:char:PV}.
% \Halmos
 \endproof

Each optimizer $x^*$ of \reff{min:max(f):K} is a weakly Pareto point.
One can solve \reff{detect:Pareto:u} to check if
$x^*$ is a Pareto point or not.
If it is not, each minimizer of \reff{detect:Pareto:u} is a Pareto point.
We remark that \reff{min:max(f):K} can be reformulated as polynomial optimization.
By introducing the new variable $x_{n+1}$,
the optimization~\reff{min:max(f):K} is equivalent to
\beq \label{find:wP:B2>fi}
\left\{ \begin{array}{rl}
\min   & x_{n+1} \\
\st    & x_{n+1} \ge f_i(x) \, (i \in  [m]),  \\
       & x \in K.
\end{array} \right.
\eeq
The Moment-SOS hierarchy can be applied to solve it.
When the set $K$ is unbounded,
the feasible set of \reff{find:wP:B2>fi} is also unbounded.
The Moment-SOS hierarchy may not converge.
For such a case, we refer to
the homogenization method in Subsection~\ref{ssc:weights}
or the method in \cite{MLM19}.

Once a minimizer $x^*$ for \reff{find:wP:B2>fi} is obtained,
we can solve \reff{detect:Pareto:u} to detect
if it is a Pareto point or not.
If it is not, we may get a Pareto point by solving \reff{detect:Pareto:u}.

\begin{example}
Consider the MOP with objectives
\[
\baray{ll}
f_1 = x_1^3+x_2^3-x_3^3+x_3^2x_4^2, \, & f_2 = x_2^3+x_3^3-x_4^3+x_4^2x_1^2, \\
f_3 = x_3^3+x_4^3-x_1^3+x_1^2x_2^2, \, & f_4 = x_4^3+x_1^3-x_2^3+x_2^2x_3^2,
\earay
\]
and with the exterior constraint $x_1^3+x_2^3+x_3^3+x_4^3 \ge 1$.
All $f_1,f_2,f_3,f_4$ are unbounded below on $K$.
The CSP~\reff{prob:CSP} does not exist since each $f_i^* = -\infty$.
However, solving \reff{find:wP:B2>fi}
%%by the Moment-SOS hierarchy,
%%With the relaxation order $k=4$,
gives the Pareto point $(0.6300,0.6300,0.6300,0.6300).$
\end{example}

%%%%%%%%%%%%%%%%%%%%%%%%%%%%%%%%%%%%%%%%%%%%%%%%%%%%%%
\iffalse

(ii) Consider the MOP with objectives
\[
\baray{l}
f_1 = -x_1^2 + x_2^2 + x_3^3 +x_4^4+ x_1x_2x_3x_4, \\
f_2 = x_1-x_2^3-x_3^4-x_4^5+x_2x_3x_4, \\
f_3 = -x_1^4 -x_2^4-x_3^4-x_4^4+ 2(x_2x_3+x_2x_4+x_3x_4)-x_1x_2x_3x_4
\earay
\]
and with the constraints
\[
1\le x_1x_2,x_1x_3,x_1x_4 \le 2, \quad
x_1 \ge 0, x_2 \ge 0, x_3 \ge 0, x_4 \ge 0.
\]
All objectives are unbounded below over the feasible set.
There does not exist a weight $w \ge 0$ such that the
LSP \reff{prob:LSP} is bounded below.
However, by solving \reff{find:wP:B2>fi}, %% with relaxation order $k=4$,
we get the Pareto point $( 1.6874, 0.5926, 1.1853, 0.5926)$.

\fi
%%%%%%%%%%%%%%%%%%%%%%%%%%%%%%%%%%%%%%%%%%%%%%%%%%%%%%%%%%%%%

\subsection{Nonexistence of WPPs}
\label{ssc:noWPP}

We discuss how to detect nonexistence of weakly Pareto points,
when $K$ is unbounded. Recall that $d_i  \coloneqq  \deg(f_i)$.
Observe that \reff{min:max(f):K} is unbounded below if and only if
the following optimization is unbounded below:
\beq \label{minmax:x0>=f:K}
\left\{ \begin{array} {rl}
\min  & x_{n+1}  \\
\st  &  -(-x_{n+1})^{d_i} - f_i(x) \ge 0 \, (i \in  [m]),  \,
       x \in K.
\end{array} \right.
\eeq
Let $K_1$ be the feasible set of \reff{minmax:x0>=f:K}
and let its homogenization be (note $\tilde{x} \coloneqq (x_0, x)$):
\beq  \label{set:K2hmg}
 \widetilde{K}_1 \,  \coloneqq
 \left\{ (x_0, x, x_{n+1}) \left|
 \begin{array} {c}
  -(-x_{n+1})^{d_i} - \widetilde{f_i}(\tilde{x})  \ge 0 \, (i \in [m]),   \\
  \widetilde{c_i}(\tilde{x}) = 0 \, (i \in \mcal{E}),  \\
  \widetilde{c_j}(\tilde{x}) \ge 0 \, (j \in \mcal{I}),  \\
 \| \tilde{x} \|^2 + \| x_{n+1}  \|^2 = 1, \, x_0 \ge 0
 \end{array} \right. \right\} .
\eeq
When $K_1$ is closed at $\infty$, $x_{n+1} \ge \gamma$ on $K_1$
if and only if $x_{n+1} - \gamma  x_0 \ge 0$ on $\widetilde{K}_1$, i.e.,
$x_{n+1} - \gamma  x_0 \in \mathscr{P}_1( \widetilde{K}_1)$.
So, we consider the linear conic optimization
\beq \label{mxgm:PdtK1}
%\left\{ \begin{array} {rl}
\max \quad \gamma \quad
\st  \quad  x_{n+1} - \gamma  x_0  \in \mathscr{P}_1( \widetilde{K}_1) .
%\end{array} \right.
\eeq
The optimization \reff{min:max(f):K} is unbounded below if and only if
\reff{mxgm:PdtK1} is infeasible, when $K_1$ is closed at $\infty$.
The dual optimization of \reff{mxgm:PdtK1} is
\beq \label{min<x0,checky>}
%\left\{ \begin{array} {rl}
\min \quad \langle x_{n+1}, \check{y} \rangle  \quad
\st  \quad  \langle x_0, \check{y} \rangle = 1, \,
       \check{y} \in \mathscr{R}_1( \widetilde{K}_1 ).
%\end{array} \right.
\eeq
Note that \reff{min<x0,checky>} is feasible if $K$ is nonempty.
So, it is unbounded below if there is a decreasing ray $\Delta \check{y}$:
\beq \label{maxx0:x0-gm:inPdtK2}
 \begin{array} {c}
 \langle x_{n+1},  \Delta \check{y} \rangle  = -1, \quad
 \langle x_0, \Delta \check{y} \rangle = 0, \quad
 \Delta \check{y} \in \mathscr{R}_1( \widetilde{K}_1 ) .
\end{array}
\eeq
Since $x_0 \ge 0$ on $\widetilde{K}_1$,
the equality $\langle x_0, \Delta \check{y} \rangle = 0$
implies that every representing measure for $\Delta \check{y}$
is supported in $x_0 = 0$.
Therefore, \reff{maxx0:x0-gm:inPdtK2} is equivalent to
\beq \label{<x0,Dty>=-1:ray}
\begin{array} {c}
 \langle x_{n+1},  \Delta \ty \rangle  = -1, \quad
    \Delta \ty \in \mathscr{R}_1( K_1^\circ ) ,
\end{array}
\eeq
where $K_1^\circ$ is the linear section $x_0 = 0$ of $\widetilde{K}_1$:
\beq  \label{set:K2o}
K_1^\circ \,  \coloneqq  \,
\left\{ (x, x_{n+1}) \left|
\begin{array} {c}
    -(-x_{n+1})^{d_i} - f_i^{hom}(x) \ge 0 \, (i \in [m]),   \\
   c_i^{hom}(x) = 0 \, (i \in \mcal{E}),  \\
   c_j^{hom}(x)\ge 0 \, (j \in \mcal{I}),  \\
   \| x \|^2  + x_{n+1}^2 = 1
\end{array} \right. \right\} .
\eeq
The following is the theorem for nonexistence of WPPs.

\begin{theorem}\label{thm:noWPP}
%%Let $K_1, K_1^\circ$ be as above.
Suppose $\Delta \ty = \lmd v$, with $\lmd >0$ and $v \in K_1^\circ$,
is a feasible point for \reff{<x0,Dty>=-1:ray}.
If the point $(0,v) \in \cl{ \widetilde{K}_1 \cap \{ x_0 > 0 \} }$,
then \reff{minmax:x0>=f:K} and \reff{min:max(f):K}
must be unbounded below, and hence there are no weakly Pareto points.
\end{theorem}
\proof
The unboundedness of \reff{minmax:x0>=f:K}
is implied by the item (ii) of Theorem~\ref{thm:unbounded:pop},
for the case that $g^\hm  \coloneqq  x_{n+1}$ and $K^\circ$ is replaced by $K_1^\circ$.
Note that \reff{min:max(f):K} is unbounded below if and only if
\reff{minmax:x0>=f:K} is unbounded below.
So, \reff{min:max(f):K} is also unbounded below.
By Theorem~\ref{thm: exist of Pareto}, there are no weakly Pareto points.
% \Halmos
  \endproof

The tms $\Delta \ty = \lmd v$ satisfying \reff{<x0,Dty>=-1:ray}
can be obtained by Algorithm~\ref{alg:ubd} with a minor variation.
The only difference is to choose a generic
$R \in \inte{ \Sigma[x,x_{n+1}]_{2d_1} }$
and then solve the hierarchy of moment relaxations:
\beq \label{min:txRtx:K2o}
\left\{ \begin{array}{cl}
\min & \langle R, z \rangle \\
\st  &  \langle x_{n+1}, z\rangle = -1,  \\
     &   L^{(k)}_{\|(x, x_{n+1})\|^2-1}[z] = 0 ,   \\
     & L^{(k)}_{c_i^{hom} }[z] = 0 \, (i \in \mcal{E}), \\
     &  L^{(k)}_{c_j^{hom} }[z] \succeq 0 \, (j \in \mcal{I}),   \\
     &  L^{(k)}_{h_i}[z] \succeq 0  \, (i \in [m]), \\
     & M_k[z] \succeq 0, \, z \in \re^{ \N^{n+1}_{2k} }.
\end{array} \right.
\eeq
In the above, each $h_i  \coloneqq -(-x_0)^{d_i} - f_i^{hom}(x)$.
The convergence property for the hierarchy of \reff{min:txRtx:K2o}
is similar to that for Theorem~\ref{thm:cvg:ubd}.

\begin{example}
Consider the MOP with objectives
\[
\baray{rcl}
f_1 &=& (x_1x_2+x_3x_4)(x_1x_4+x_2x_3)+x_1^2+x_2^2+x_3^2+x_4^2,  \\
f_2 &=&  x_1^3x_2^2 + x_2^3x_3^2 + x_3^3x_4^2+x_4^3x_1^2,  \\
f_3 &=&  x_1^4-x_2^4+x_3^4-x_4^4+x_1x_2x_4+x_1x_3x_4,   \\
f_4 &=& (x_1-x_2)(x_3-x_4)^2+(x_1-x_3)(x_2-x_4)^2+ \\
          & & \quad (x_1-x_4)(x_2-x_3)^2+  x_1x_2+x_2x_3+x_3x_4,
\earay
\]
and with the constraints $x_1x_2x_3 \ge 1$, $x_2x_3x_4 \ge 1$.
Solving the moment relaxation~\reff{min:txRtx:K2o}
gives the feasible point $\Dt \ty = 3.3597[v]_5$ with
\[
v = (v_1, v_2, v_3, v_4, v_5) =
(-0.2761, 0.8737,  0.0000,   -0.2680, -0.2976).
\]
The set $K_1$ is not closed at infinity, but $(0,v)$ still belongs to
$cl \big( \widetilde{K}_1 \cap \{ x_0 >0 \} \big)$.
This is implied by Lemma~\ref{lemma: cond closure},
since $\Delta x=(0,0,-1,0,0)^T$ satisfies the condition \eqref{eq:direction condition}.
By Theorem~\ref{thm:noWPP}, there is no weakly Pareto point.
\end{example}

\subsection{Nonexistence of PPs}
\label{ssc:noPP}

When there are no weakly Pareto points,
there must exist no Pareto points.
So Theorem~\ref{thm:noWPP} is also applicable to
detect nonexistence of Pareto points.
However, a Pareto point may not exist while weakly Pareto points exist.
This section discusses how to detect nonexistence of
Pareto points for this case.

We consider the optimization \reff{detect:Pareto:u} with $x^* \in K$.
A Pareto point exists if and only if \reff{detect:Pareto:u}
is bounded below and has a minimizer for some $x^*\in K$.
The ``if" implication is clear. When $x^*$ itself is a Pareto point, then
$x^*$ must be a minimizer for \reff{detect:Pareto:u}.
This explains the ``only if" implication.
Let $K(x^*)$ be the feasible set of \reff{detect:Pareto:u}
determined by $x^*$ and
let $\widetilde{K}(x^*)$ be the homogenization of $K(x^*)$
similarly as in \eqref{set:tldK}.
Suppose $K(x^*)$ is closed at $\infty$.
Then \reff{detect:Pareto:u} is bounded below if and only if
$\widetilde{ f_e }(\tilde{x}) - \gamma  x_0^{d}
\in \mathscr{P}_d( \widetilde{K}(x^*))$ for some $\gamma$.
We consider the linear conic optimization
\beq \label{max:gamma in Ky}
%\left\{ \begin{array} {rl}
\max \quad \gamma \quad
\st  \quad \widetilde{ f_e }(\tilde{x})  - \gamma  x_0^{d}
\in \mathscr{P}_d( \widetilde{K}(x^*)) .
%\end{array} \right.
\eeq
Pareto points do not exist if \reff{detect:Pareto:u}
is unbounded below for all $x^*\in K$. This is equivalent to that
\reff{max:gamma in Ky} is infeasible for all $x^*\in K$.
The dual optimization of \reff{max:gamma in Ky} is
\beq \label{min<sumf,checky>}
%\left\{ \begin{array} {rl}
\min \quad \langle \widetilde{ f_e }, \tilde{y} \rangle  \quad
\st \quad \langle x_0^{d}, \tilde{y} \rangle = 1, \,
       \tilde{y} \in \mathscr{R}_d( \widetilde{K}(x^*) ).
%\end{array} \right.
\eeq
By weak duality, \reff{max:gamma in Ky} is infeasible if \reff{min<sumf,checky>}
is unbounded below. The problem \reff{min<sumf,checky>}
is feasible for all $x^*\in K$.
Therefore, \reff{min<sumf,checky>} is unbounded below if
there is a decreasing ray $\Delta \tilde{y}$:
\beq \label{7.24}
\langle  \widetilde{ f_e }, \Delta \tilde{y} \rangle   = -1, \quad
\langle x_0^{d}, \Delta \tilde{y} \rangle = 0, \quad
\Delta \tilde{y} \in \mathscr{R}_d( \widetilde{K}(x^*) ).
\eeq
Since $x_0 \ge 0$ on $\widetilde{K}(x^*)$,
$\langle x_0^{d}, \Delta \check{y} \rangle = 0$
if and only if every representing measure for
$\Delta \check{y}$ is supported in the hyperplane $x_0=0$.
Hence, the existence of $\Delta \tilde{y}$
satisfying \reff{7.24} is equivalent to
the existence of $\Delta \check{y}$ satisfying
\beq \label{<x0,sumf>=-1}
 \begin{array} {c}
 \langle f_e^{hom},  \Delta \check{y} \rangle  = -1, \quad
 \Delta \check{y} \in \mathscr{R}_d( \widetilde{K}_0^* ).
\end{array}
\eeq
where $f_e^{hom}(x) \coloneqq \widetilde{ f_e }(0,x)$ and $K_0^*$
is the section $x_0=0$ of $K(x^*)$:
\beq  \label{set:K^*_0}
 K_0^*\,  \coloneqq  \,
 \left\{ x \left|
 \begin{array} {c}
    c_j^{hom}(x)  = 0 \, (j \in \mcal{E}),  \\
    c_j^{hom}(x) \ge 0 \, (j \in \mcal{I}),  \\
   - f_i^{hom}(x)  \ge 0 \, (i \in [m]),  \\
    x^Tx = 1.
 \end{array} \right. \right\} .
\eeq
It is important to observe that $K^*_0$ and \reff{<x0,sumf>=-1}
do not depend on $x^*$. If there exists
$\Delta \check{y}$ satisfying \reff{<x0,sumf>=-1},
then \reff{detect:Pareto:u}
is unbounded below for all $x^*\in K$,
and hence there are no Pareto points.
This implies the following theorem.

\begin{theorem}\label{thm:noPPs}
Suppose $K(x^*)$ is closed at infinity for all $x^*\in K$.
If there is $\Delta \check{y}$ satisfying \reff{<x0,sumf>=-1},
then \reff{detect:Pareto:u} is unbounded below for all $x^*\in K$
and hence Pareto points do not exist.
\end{theorem}

Theorem \ref{thm:noPPs} only shows nonexistence of Pareto points,
but it does not imply nonexistence of weakly Pareto points. For instance,
consider the MOP
\[
\left\{ \baray{rl}
\min  &  (x_1, x_2)  \\
\st   &  x_1 \ge 0.
\earay \right.
\]
The tms
$\Delta \tilde{y}  \coloneqq  [(0,-1)]_1 $ satisfies \reff{<x0,sumf>=-1},
so there are no Pareto points.
But each $(0,x_2)$ is a weakly Pareto point.
The existence of $\Delta \tilde{y}$ satisfying \reff{<x0,sumf>=-1}
can be checked by applying Algorithm~\ref{alg:ubd} similarly,
with the polynomial $g_1  \coloneqq  f_e^\hm$ and the set $K_0^*$.
The properties are summarized in
Theorems~\ref{thm:unbounded:pop} and \ref{thm:cvg:ubd}.

\begin{example}
Consider the objectives
\[
\baray{rcl}
f_1 &=& x_1^4 + x_3^4 + (x_1x_2)^2+ (x_2x_3)^2+ (x_3x_4)^2 + x_1x_2x_3x_4, \\
f_2 &=& x_1^4+ x_2^4+x_3^4+x_4^4 -2x_2^4 - x_1^3x_2 - x_3^3x_4,
\earay
\]
and the constraint $x_1x_2x_3x_4 \ge 0$.
Since $f_1(0,t,0,0)=0 $ is the minimum value,
the point $(0,t,0,0)$ is a weakly Pareto point for all $t\in \re$.
Since all the polynomials are homogeneous,
$K(x^*)$ is closed at infinity for all $x^* \in K$.
By Algorithm~\ref{alg:ubd}, we get
$\Delta \tilde{y} =1.0023[u]_4$ satisfying \reff{<x0,sumf>=-1},
for $u=(0.0000,-0.9994,0.0000,0.0339)$.
Hence, there is no Pareto point.
\end{example}

\section{Conclusions and discussions}
\label{sc:con}

This paper studies multi-objective optimization given by polynomials.
We characterize the convex geometry for (weakly) Pareto values
and give convex representations for them.
For LSPs, we show how to use tight relaxations to solve them,
how to find proper weights, and how to detect
nonexistence of proper weights.
For CSPs, we show how to solve them by moment relaxations.
Furthermore, we show how to check if a given point is a
(weakly) Pareto point and how to detect
existence or nonexistence of (weakly) Pareto points.
To detect nonexistence of proper weights and (weakly) Pareto points,
we also show how to detect unboundedness of polynomial optimization.

There are some open questions for studying these topics.
To detect nonexistence of (weakly) Pareto points,
or to detect nonexistence of proper weights,
we need to check unboundedness of polynomial optimization.
This is discussed in Section~\ref{sc:unbounded}.
A feasible point for the system \reff{<ghm,z>=-1:Rd(tK)}
is only a sufficient condition for unboundedness of
the optimization \reff{pop:unbounded:fcij}, but it may not be necessary.

\begin{question}
When \reff{<ghm,z>=-1:Rd(tK)} is infeasible,
what is a computationally convenient certificate for
unboundedness of \reff{pop:unbounded:fcij}?
\end{question}

Another important question is to detect nonexistence of proper weights.
This is discussed in Subsection~\ref{ssc:noprop:weight}.
We have seen that \reff{-1:Dy:RdK}
is sufficient for the proper weight set $\mcal{W} = \emptyset$,
but it may not be necessary.

\begin{question}
When \reff{-1:Dy:RdK} does not have a feasible point,
how can we detect nonexistence of proper weights?
\end{question}

In Subsections~\ref{ssc:noWPP} and \ref{ssc:noPP},
we discussed how to detect nonexistence of (weakly) Pareto points.
Under certain conditions, we have shown that
\reff{<x0,Dty>=-1:ray} implies nonexistence of weakly Pareto points
and \reff{<x0,sumf>=-1} implies nonexistence of Pareto points.
However, they may not be necessary for nonexistence.

\begin{question}
Beyond \reff{<x0,Dty>=-1:ray} and \reff{<x0,sumf>=-1},
what are computationally convenient certificates for
nonexistence of (weakly) Pareto points?
\end{question}

The above questions are mostly open, to the best of the authors' knowledge.
They are interesting future work.

\bigskip
\noindent
{\bf Acknowledgement}
Jiawang Nie is partially supported by the NSF grant
DMS-2110780.

\appendix
%\appendixpage
%\addappheadtotoc

\section{Unboundedness in Polynomial Optimization}
\label{sc:unbounded}

This section discusses how to detect unboundedness of
a polynomial optimization problem.
This question is very important for detecting nonexistence
of proper weights and (weakly) Pareto points,
in Section~\ref{sc:LSP} and Section~\ref{sc:detect}.

For a polynomial $g(x)$ of degree $d$, consider the optimization
\beq \label{pop:unbounded:fcij}
%\left\{\baray{rl}
\inf \quad g(x) \quad
\st \quad x \in K.
%\earay \right.
\eeq
The feasible set $K$ is the same as for \reff{prob:mop}.
When $K$ is unbounded, \reff{pop:unbounded:fcij}
may be unbounded below, i.e., there exists a sequence
$\{u_k\} \subseteq K$ such that $g(u_k) \to -\infty$.
We discuss how to detect unboundedness of \reff{pop:unbounded:fcij}.
{Equivalently, the problem \reff{pop:unbounded:fcij}
is unbounded below if and only if
\[
\inf \{g(x)|x\in K \} \, = \, -\infty .
\]}
%For the feasible set $K$ of \reff{prob:mop},
%%%$K \coloneqq \{x|p_i(x)\ge0 \,(i \in \mathcal{E}),p_j(x)=0 \,(j\in \mathcal{I})\}$,
The homogenization of the set $K$ is
($\tx  \coloneqq   (x_0, x)$ is the homogenizing variable)
\beq \label{set:tldK}
\widetilde{K} \,  \coloneqq  \, \left\{
\tx %%= (x_0, x)
\left|\baray{rcl}
 \tilde{c}_i(\tx) &=& 0 \, (i \in \mcal{E}), \\
 \tilde{c}_j(\tx) &\ge& 0 \,(j \in \mcal{I}),\\
 \tx^T\tx &=& 1, \,  x_0 \ge 0
 \earay \right. \right\},
\eeq
where $\tilde{c}_i(\tx)=x_0^{deg(c_i)}c_i(x/x_0)$
is the homogenization of $c_i(x)$. The ball constraint $\tilde{x}^T\tilde{x}=1$
is added to make the set $\tilde{K}$ compact.
The constraint $x_0\ge 0$ ensures that
$\tilde{g}(\tilde{x})-\gamma x_0^{deg(g)} \ge 0$ on $\tilde{K}$
implies that $g(x)-\gamma \ge 0$ on $K$. The set $K$ is said to be
{\it closed at $\infty$} (see \cite{NieDis12}) if
\[
\widetilde{K} \, = \,
\cl{ \big\{ \tx \in \widetilde{K}:  \, x_0 > 0 \big\} } .
\]
The closeness of $K$ at $\infty$
is a genericity condition, as shown in \cite{guo2014minimizing}.
When $K$ is closed at $\infty$, the polynomial $g(x)-\gamma$
is nonnegative on $K$ if and only if its homogenization
$\tilde{g}(\tilde{x})-\gamma x_0^{deg(g)}$ is nonnegative on $\tilde{K}$.

The intersection of $\tilde{K}$ and $x_0=0$ is
\beq \label{set:K0:unbdPOP}
K^\circ \,  \coloneqq  \, \left\{ x \in \re^n
\left| \baray{rcl}
c_i^\hm(x) &=& 0 \, (i \in \mcal{E} ) , \\
c_j^\hm(x) &\ge & 0  \, (j \in \mcal{I} ),   \\
 x^Tx &=& 1,  \\
\earay\right. \right\},
\eeq
where each $c_i^\hm(x) = \tilde{c}_i(0,x)$.

\subsection{A certificate for unboundedness}
\label{ssc:cert:ubd}

The optimization \reff{pop:unbounded:fcij} is bounded below
if and only if $g$ has a lower bound $\gamma$ on $K$, i.e.,
$g -\gamma \in \mathscr{P}_d(K)$. So we consider the optimization
\beq  \label{max:gm:f-gm>=0K}
\left\{\baray{rl}
\max &  \gamma   \\
\st  &  g  - \gamma  \in \mathscr{P}_d(K).
\earay \right.
\eeq
To check infeasibility of \reff{max:gm:f-gm>=0K},
we use the homogenization trick in \cite{NieDis12}.
When $K$ is closed at $\infty$,
a polynomial $p \ge 0$ on $K$ if and only if its homogenization
$\tilde{p} \ge 0 $ on $\widetilde{K}$ (see \cite{HNY1,NieDis12}).
So, the membership $g -\gamma \in \mathscr{P}_d(K)$ is equivalent to
$\widetilde{g} -\gamma x_0^d \in \mathscr{P}_d(\widetilde{K})$,
and hence \reff{max:gm:f-gm>=0K} is the same as
\beq  \label{max:gm:f-gm>=0:hmgK}
\left\{\baray{rl}
\max & \gamma   \\
\st  &  \widetilde{g} -\gamma x_0^d \in
\mathscr{P}_d \big( \widetilde{K} \big).
\earay \right.
\eeq
The dual optimization of \reff{max:gm:f-gm>=0:hmgK} is
\beq  \label{min<x0^d,y>:mom:tK}
\left\{\baray{rl}
\min & \langle \widetilde{g}, y \rangle  \\
\st  &  \langle x_0^d, y \rangle = 1, \,
        y \in  \mathscr{R}_d \big( \widetilde{K} \big).
\earay \right.
\eeq
If \reff{min<x0^d,y>:mom:tK} is unbounded below, then
\reff{max:gm:f-gm>=0:hmgK} must be infeasible,
which implies that \reff{max:gm:f-gm>=0K}
is infeasible and \reff{pop:unbounded:fcij} is unbounded below,
when $K$ is closed at $\infty$.

When $K \ne \emptyset$, the linear conic optimization \reff{min<x0^d,y>:mom:tK}
has a feasible point. It is unbounded below
if there is a decreasing ray $\Dt y$:
\beq  \label{<ghm,Dy>=-1:mom}
\langle \widetilde{g}, \Dt y \rangle = -1, \quad
\langle x_0^d, \Dt y \rangle = 0, \quad
\Dt y \in \mathscr{R}_d \big( \widetilde{K} \big).
\eeq
If $\nu$ is a representing measure for $\Dt y$
and is supported in $\widetilde{K}$, then
\[
0 = \langle x_0^d, \Dt y \rangle = \int x_0^d \mt{d} \nu
\]
implies that
$
\supp{\nu} \subseteq   \widetilde{K} \cap \{ x_0 = 0 \}.
$
Thus, \reff{<ghm,Dy>=-1:mom} is equivalent to
\beq  \label{<ghm,z>=-1:Rd(tK)}
\langle g^\hm, z \rangle = -1, \quad
z \in \mathscr{R}_d( K^\circ ),
\eeq
where $K^\circ$ is the set as in \eqref{set:K0:unbdPOP}.
Let $d_1$ be the degree
\beq \label{deg:d1:ubd}
d_1 \,  \coloneqq  \,   \lceil d/2 \rceil  .
%d_1 \,  \coloneqq  \, \max \big\{ \lceil d/2 \rceil,
%\lceil \deg(c_i)/2 \rceil \, (i \in \mcal{E} \cup \mcal{I})  \big\} .
\eeq
To check if \reff{<ghm,z>=-1:Rd(tK)} is feasible or not,
we select a generic $R \in \inte{ \Sig[x]_{2d_1} }$ and
consider the linear moment optimization
\beq  \label{min<R,z>:gz=-1}
\left\{ \baray{rl}
\min & \langle R, z \rangle  \\
\st  & \langle g^\hm, z \rangle = -1, \,
       z \in \mathscr{R}_{2d_1}( K^\circ ) .
\earay \right.
\eeq
The following shows how to detect
unboundedness of \reff{pop:unbounded:fcij}.

\begin{theorem} \label{thm:unbounded:pop}
Let $\widetilde{K}$, $K^\circ$ be the sets as in
\reff{set:tldK}-\reff{set:K0:unbdPOP}.

\bit

\item [(i)] Suppose \reff{<ghm,z>=-1:Rd(tK)} is feasible.
If $R \in \inte{ \Sig[x]_{2d_1} }$ is generic, then \reff{min<R,z>:gz=-1}
has a unique optimizer $z^*$ and $z^* = \lmd [u]_{2d_1}$,
with $u \in K^\circ$ and $\lmd > 0$.

\item [(ii)] %%Assume $K \ne \emptyset$ and
Suppose $z  \coloneqq  \lmd [u]_{d}$, with $u \in K^\circ$ and $\lmd > 0$,
is a feasible point for \reff{<ghm,z>=-1:Rd(tK)}.
If the point $(0,u) \in \cl{ \widetilde{K} \cap \{ x_0 > 0 \} }$,
then \reff{pop:unbounded:fcij} is unbounded below.

\eit
\end{theorem}
\proof
(i) Since $R$ is generic in the interior $\inte{ \Sig[x]_{2d_1} }$,
there exists $\eps >0$ such that
\[
R - \eps \|[x]_{d_0}\|^2 \in \Sig[x]_{2d_1} .
\]
Hence, for all $z \in \mathscr{R}_{2d_0}( K^\circ )$, it holds that
\[
\langle R , z \rangle \, \ge \,
\eps \langle \|[x]_{d_1}\|^2 , z \rangle \, \ge  \,
\eps \cdot \tr{ M_{d_1}[z] }.
\]
Since \reff{<ghm,z>=-1:Rd(tK)} is feasible,
the optimization~\reff{min<R,z>:gz=-1}
is also feasible, say, $z^{(0)}$ is a feasible point.
Then, \reff{min<R,z>:gz=-1} is equivalent to
\beq  \label{min<R,z>:unbd:sublevel}
\left\{ \baray{rl}
\min & \langle R, z \rangle  \\
\st &   \tr{ M_{d_1}[z] } \le \frac{1}{\eps} \langle R , z^{(0)} \rangle , \\
   & \langle f^\hm, z \rangle = -1, \\
   &  z \in \mathscr{R}_{2d_1}( K^\circ ) .
\earay \right.
\eeq
The feasible set of \reff{min<R,z>:unbd:sublevel} is compact,
so it has an optimizer, say, $z^*$.
When $R$ is generic, the optimizer $z^*$ must be unique
and it is an extreme point of the feasible set of \reff{min<R,z>:gz=-1}.
Since \reff{min<R,z>:gz=-1} has only a single equality constraint,
the optimizer $z^*$ must lie in an extreme ray of the cone
$\mathscr{R}_{2d_1}( K^\circ )$.
This means that $z^* = \lmd [u]_{2d_1}$
for a point $u \in K^\circ$ and a scalar $\lmd > 0$
(note $z^*$ is nonzero).
%%%%%%%%%%%%%%%%%%%%%%%%%%%%%%%%%%%%%%%%%%%%%%%%%%%%%%%%%%%%%%%%%%%%%
%%%%%%%%%%%%%%%%%%%%%%%%%%%%%%%%%%%%%%%%%%%%%%%%%%%%%%%%%%%%%%%%%%%%%%
% \iffalse

% Suppose $z^*$ has the decomposition
% \[
% z^* \,=\, \lmd_1 [u_1]_{2d_1} + \cdots + \lmd_r [u_r]_{2d_1}
% \]
% for distinct points $u_1, \ldots, u_r \in K^\circ$
% and positive scalars $\lmd_i>0$.
% Then we consider the following linear program
% \beq  \label{min:nu:z*dcomp:unbd}
% \left\{ \baray{cl}
% \min\limits_{(\tau_1, \ldots, \tau_r)} &
%     \tau_1 R(u_1) + \cdots + \tau_r R(u_r)  \\
% \st & \tau_1 g^\hm(u_1) + \cdots + \tau_r f^\hm(u_r)  = -1, \\
%   & \tau_1 \ge 0, \ldots, \tau_r \ge 0.
% \earay \right.
% \eeq
% Note that \reff{min<R,z>:gz=-1} and
% \reff{min:nu:z*dcomp:unbd} have the same optimal value.
% By the fundamental theorem of LP, the above has an optimizer
% that is a basic feasible solution. That is, it has a minimizer
% $\tau^* = (\tau_1^*, \ldots, \tau_r^*)$ has only one nonzero entry,
% say, $\tau_1^* >0$ and $\tau_2^* = \cdots = \tau_r^*=0$.
% This implies that $\tau_1^* [u_1]$ is also a minimizer of
% \reff{min<R,z>:gz=-1}. Since the optimizer $z^*$ is unique,
% we have $z^* = \tau_1^* [u_1]$, so the conclusion is true.

% \fi
%%%%%%%%%%%%%%%%%%%%%%%%%%%%%%%%%%%%%%%%%%%%%%%%%%%%%%%%%%%%%%%%%%%%%%
%%%%%%%%%%%%%%%%%%%%%%%%%%%%%%%%%%%%%%%%%%%%%%%%%%%%%%%%%%%%%%%%%%%%%%

(ii)
Since $(0,u) \in cl \big( \widetilde{K} \cap \{ x_0 > 0 \} \big)$,
there is a sequence
\[
\{(t_k, u_k)\}_{k=1}^\infty \subseteq \widetilde{K}\cap \{x_0>0\}
\]
such that $\lim\limits_{k \to \infty} (t_k, u_k) = (0, u)$.
Note that each $t_k > 0$ and
\[
-1 = \langle g^{hom}, \lmd [u]_d \rangle
= \lmd g^{hom}(u) =  \lim_{k \to \infty}
 \lmd \tilde{g}(t_k, u_k) .
\]
Thus, for $k$ big enough,
$ \lmd \tilde{g}(t_k, u_k) \le -\frac{1}{2}$ and
\[
 \lmd \tilde{g}(t_k, u_k)  = \lmd \cdot (t_k)^d  g(u_k/t_k) \le -1/2  .
\]
This implies that $g(u_k/t_k) \le \frac{-1}{2 \lmd \cdot (t_k)^d }$
for all $k$ big enough, so $g(u_k/t_k) \to -\infty$ as $k \to \infty$.
Since each $u_k/t_k \in K$, $g$ is unbounded below on $K$.
% \Halmos
\endproof

In computational practice,
the generic polynomial $R \in \inte{ \Sig[x]_{2d_1} }$ can be chosen as
$[x]_{d_1}^TA^TA[x]_{d_1}$, for some randomly generated square matrix $A$.

We remark that the closeness of $K$ at $\infty$
is a generic condition, as shown in \cite{guo2014minimizing}.
In Theorem~\ref{thm:unbounded:pop}(ii), we use the relaxed condition $(0,u)\in cl(\widetilde{K}\cap \{x_0>0\})$ instead of the closeness at $\infty$. For the relaxed condition, we give a sufficient condition in Lemma \ref{lemma: cond closure} to check if it is satisfied.

\begin{lemma} \label{lemma: cond closure}
    Let $\widetilde{K}$, $K^\circ$ be the sets as in
\reff{set:tldK}-\reff{set:K0:unbdPOP} and $z  \coloneqq  \lmd [u]_{d}$, with $u \in K^\circ$ and $\lmd > 0$,
be a feasible point for \reff{<ghm,z>=-1:Rd(tK)}.
If there exist $\Delta x \in \mathbb{R}^n$ and $\delta_0>0$ such that
\begin{equation}\label{eq:direction}
  c_i^{hom}(u+t\Delta x) >0 \quad \forall \, t \in (0, \delta_0), \,
  \forall \, i \in \mathcal{C} \coloneqq
  \{i\in \mathcal{E}\cup \mathcal{I}|c_i^{hom}(u)=0 \} ,
\end{equation}
% \begin{equation}\label{eq:direction}
%     \nabla c_i^{hom}(u)^T \Delta x>0 \quad \forall i \in \mathcal{C},
% \end{equation}
%where $\mathcal{C}=\{i\in \mathcal{E}\cup \mathcal{I}|c_i^{hom}(u)=0 \}$,
then $(0,u)\in cl(\widetilde{K}\cap \{x_0>0\})$.
\end{lemma}
\proof
The constraint polynomial $\tilde{c}_i(x_0,x)$ can be rewritten as
\[
  \tilde{c}_i(x_0,x) = c_i^{hom}(x) + x_0h_i(x_0,x),
\]
for some polynomial $h_i(x_0,x)$. When $i \in \mathcal{C}$, it satisfies the condition \eqref{eq:direction}. When $i \notin \mathcal{C}$, it holds that $c_i^{hom}(u)>0$. Therefore, there are $M>0$ and $0<\delta<\delta_0$ such that
\[
c_i^{hom}(u+t\Delta x)>0 \quad \text{and} \quad  h_i(x_0,x)>-M
\]
for all $\delta>x_0,t>0$ and $i\in \mathcal{E}\cup \mathcal{I}$.
Let $\{t_k\}_{k=1}^{\infty}$ be a sequence such that
$\lim\limits_{k\to\infty} t_k = 0$ and $\delta>t_k>0$ for all $k$.
For each $k$, we define
\[
s_k  \, \coloneqq \,
\min\Big(\frac{\delta}{2k},
\{\frac{c_i^{hom}(u+t_k\Delta x)}{2M}\}_{i\in \mathcal{E}\cup \mathcal{I}}
\Big)>0.
\]
For all $i\in \mathcal{E}\cup \mathcal{I}$, it holds that
 \begin{align*}
  \tilde{c}_i(s_k,u+t_k\Delta x) &= c_i^{hom}(u+t_k \Delta x) + s_kh_i(s_k,u+t_k\Delta x) \\
  &\ge c_i^{hom}(u+t_k \Delta x) + \frac{c_i^{hom}(u+t_k\Delta x)}{2M} (-M) \\
  &\ge \frac{1}{2}c_i^{hom}(u+t_k \Delta x) \\
  &> 0.
 \end{align*}
For convenience, we denote
\[
\tilde{u}_k \coloneqq (s_k,u+t_k\Delta x)/\|(s_k,u+t_k\Delta x)\| .
\]
Each $\tilde{c}_i$ is homogeneous, so $\tilde{c}_i(\tilde{u}_k)>0$
by above inequalities.
It implies $\tilde{u}_k \in \widetilde{K}$.
The construction of sequences ensures
\[
\lim\limits_{k\to\infty} s_k = \lim\limits_{k\to \infty} t_k=0.
\]
Thus, $\lim\limits_{k\to \infty} \tilde{u}_k = (0,u)$.
So, it shows that $(0,u)\in cl(\widetilde{K}\cap \{x_0>0\})$.
% \Halmos
\endproof

The sufficient condition \eqref{eq:direction} in Lemma \ref{lemma: cond closure}
requires that $\Delta x$ is an increasing direction for $c_i^{hom}(x)$ at $x=u$
for $i\in\mathcal{C}$. It can be checked numerically by gradients and Hessian matrices.
We denote $\mathcal{C}_0 = \{i\in \mathcal{C}|\nabla c_i^{hom}(u)= 0\}$
and $\mathcal{C}_1 = \{i\in \mathcal{C}|\nabla c_i^{hom}(u)\neq 0\}$.
The direction $\Delta x$ satisfies the condition in \eqref{eq:direction}
if it satisfies
\beq  \label{eq:direction condition}
\left\{ \baray{r}
  \Delta x^T  \nabla^2 c_i^{hom}(u) \Delta x>0 \quad
  \forall i\in \mathcal{C}_0,   \\
  \nabla c_i^{hom}(u)^T \Delta x>0 \quad
  \forall i\in \mathcal{C}_1.
  \earay \right.
\eeq
It can be formulated as the following quadratic optimization problem
\beq \label{eq:quadratic condition}
\left\{ \baray{cl}
  \max\limits_{\Delta x, a} &  \quad  a    \\
  \text{s.t.} & \Delta x^T  \nabla^2 c_i^{hom}(u) \Delta x\ge a \quad \forall i\in \mathcal{C}_0,\\
 &   \nabla c_i^{hom}(u)^T \Delta x \ge a \quad \forall i\in \mathcal{C}_1,   \\
 &    \| \Delta x \|^2  \le 1 .
\earay \right.
\eeq
There exists a direction $\Delta x$ satisfying \eqref{eq:direction condition}
if and only if the problem \eqref{eq:quadratic condition} has the maximum $a^*>0$.
The problem \eqref{eq:quadratic condition} can be solved
as a polynomial optimization problem.
%For the case that $\mathcal{C}_0=\emptyset$,
%it is reduced to a linear program \cite{Bertsimas1997linear}.

\begin{example}
Consider the following optimization problem
\beq \nn
\left\{ \baray{rl}
\min  &  g(x)  \coloneqq  x_1^2+x_2^2+x_3^2 +  x_1x_2x_3   \\
\st &  c(x)\coloneqq x_1^2x_2^2(x_1^2+x_2^2)+x_3^6-3x_1^2x_2^2x_3^2 -1 = 0.
\earay \right.
\eeq
Note that $g^\hm =  x_1x_2x_3$ and
a feasible point of \reff{<ghm,z>=-1:Rd(tK)} is the tms
$3\sqrt{3}[u]_6$, for $u=\frac{1}{ \sqrt{3} }(1,1,-1)$.
One can check that $\nabla c^{hom}(u)=0$ and
$e^T\nabla^2 c^{hom}(u) e>0$ for $e=(1,1,1)^T$.
It demonstrates $(0, u)$ lies on the closure
$cl \big( \tilde{K} \cap \{ x_0 > 0 \} \big)$,
so this optimization problem is unbounded below.
\end{example}

When \reff{min<x0^d,y>:mom:tK} is unbounded below,
it is not necessary that \reff{min<x0^d,y>:mom:tK} has a decreasing ray,
i.e., the system \reff{<ghm,z>=-1:Rd(tK)} may be infeasible.
That is, \reff{<ghm,z>=-1:Rd(tK)} is sufficient for
unboundedness of \reff{pop:unbounded:fcij}, but it may not be necessary.
For instance, consider the optimization
\beq \nn
\left\{ \baray{rl}
\min  &  g(x)  \coloneqq  x_1x_2x_3 + x_1^2x_2^2(x_1^2+x_2^2)+x_3^6-3x_1^2x_2^2x_3^2  \\
\st & x_1^2+x_2^2 -2 x_3^2 = 0,\, x_1 x_2 \ge 0.
\earay \right.
\eeq
It is unbounded below, because
$ g(t,t, -t) = -t^3 \to -\infty$ as $t\to +\infty$,
while $(t, t, -t)$ is feasible for all $t \ge 0$.
However, the certificate \reff{<ghm,z>=-1:Rd(tK)} is infeasible.
This is because
\[
g^\hm  \, =  \, x_1^2x_2^2(x_1^2+x_2^2)+x_3^6-3x_1^2x_2^2x_3^2
\]
is the Motzkin polynomial and $\langle g^\hm, z \rangle \ge 0$
for all $z \in \mathscr{R}_d(K^\circ)$.
When \reff{<ghm,z>=-1:Rd(tK)} fails to be feasible,
the question of detecting unboundedness of \reff{pop:unbounded:fcij}
is mostly open.
%
%to the best of the author's knowledge.
%This is important future work.
%

\subsection{Solving linear moment systems}
\label{ssc:momsys}

Semidefinite relaxations can be applied to solve
\reff{<ghm,z>=-1:Rd(tK)} and \reff{min<R,z>:gz=-1}.
For more generality, we consider the moment system
\beq  \label{-1>=giz:several}
a_i \ge \langle g_i, z \rangle  \,\,(i=1,\ldots, m), \quad
z \in \mathscr{R}_d( K^\circ ),
\eeq
for given polynomials $g_1, \ldots, g_m \in \re[x]_d$
and given scalars $a_1, \ldots, a_m \in \re$.
It is worthy to note that in \reff{<ghm,z>=-1:Rd(tK)} and \reff{min<R,z>:gz=-1},
the equality is {equivalent} to the inequality like the above,
due to the conic membership condition.

Select a generic $R \in \inte{ \Sig[x]_{2d_1} }$
and consider the moment optimization
\beq  \label{<Rz>:momopt}
\left\{ \begin{array}{rl}
\min  & \langle R, z \rangle \\
\st  &  a_i - \langle g_i, z\rangle \ge 0 \,(i \in [m]), \\
     &  z \in \mathscr{R}_{2d_1}( K^\circ ) .
\end{array} \right.
\eeq
Let $d_2  \coloneqq  \max\{d_1, d_c\}$, where $d_c$ is as in \reff{deg:dc}.
For $k= d_2, d_2+1, \cdots$,
we solve the hierarchy of semidefinite relaxations
\beq  \label{min<Ry>:mom}
\left\{ \begin{array}{rl}
\min  & \langle R, z \rangle \\
\st  & a_i -\langle g_i, z\rangle \ge 0 \, \,(i=1,\ldots,m),  \\
     &  L^{(k)}_{c_i^{hom} }[z] = 0 \, (i \in \mcal{E}),   \\
     & L^{(k)}_{c_j^{hom} }[z] \succeq 0 \, (j \in \mcal{I}),   \\
     & L^{(k)}_{x^Tx-1}[z] = 0, \\
     & M_k[z] \succeq 0, \, z \in \re^{ \N^n_{2k} } .
\end{array} \right.
\eeq
Suppose $z^{(k)}$ is a minimizer of \reff{min<Ry>:mom} for a relaxation order $k$.
If there is an integer $t \in [d_c, k]$ such that
the rank condition \reff{cond:FEC} holds,
% \beq \label{ubd:rank:FT}
% \rank \, M_t[z^{(k)}] \, = \, \rank \, M_{t-d_c}[z^{(k)}],
% \eeq
then the truncation $z^{(k)}|_{2t}$ has a $r$-atomic
representing measure supported in $K^\circ$, i.e.,
\[
z^{(k)}\big|_{2t} = \lmd_1 [u_1]_{2t} + \cdots + \lmd_r [u_r]_{2t}
\]
for scalars $\lmd_1, \ldots, \lmd_r >0$,
distinct points $u_1, \ldots, u_r \in K^\circ$
and $r = \rank \, M_t[z^{(k)}]$.
Then, the truncation $z^{(k)}|_d$ is a feasible point for \reff{-1>=giz:several}.
%The following is the algorithm.

\begin{alg} \label{alg:ubd}
%%The Moment-SOS hierarchy for solving \reff{min<R,z>:gz=-1}.
Let $k  \coloneqq  d_2$. Do the following loop:
\bit

\item [Step~1]
Solve the semidefinite relaxation \reff{min<Ry>:mom}
for a minimizer $z^{(k)}$.

\item [Step~2]
Check if there exists $t \in [d_c, k]$ such that \reff{cond:FEC} holds.
If it does, then the truncation $z^{(k)}|_d$
is a feasible point for \reff{<ghm,z>=-1:Rd(tK)}.

\item [Step~3]
If \reff{cond:FEC} fails for all $t \in [d_c, k]$,
let $k  \coloneqq  k+1$ and go to Step~1.

\eit
\end{alg}

Algorithm~\ref{alg:ubd} can be implemented in
the software {\tt GloptiPoly~3} \cite{GloPol3}.
The following is the convergence property
for the hierarchy of relaxations \reff{min<Ry>:mom}.

\begin{theorem}  \label{thm:cvg:ubd}
Assume the system \reff{-1>=giz:several} is feasible
and $R \in \inte{ \Sig[x]_{2d_1} }$ is generic.
Then, we have:
\bit

\item [(i)] The optimization \reff{<Rz>:momopt} has a unique minimizer $z^*$ and
\beq \label{atomic:z*}
z^*  = \lmd_1 [u_1]_{2d_1} + \cdots + \lmd_r [u_r]_{2d_1}
\eeq
for scalars $\lmd_1, \ldots, \lmd_r >0$,
distinct points $u_1, \ldots, u_r \in K^\circ$
and $r \le m$.

\item [(ii)]
For each fixed $t \ge d_1$,
the sequence $\{ z^{(k)}|_{2t} \}_{k=d_2}^\infty$ is bounded
and every accumulation point $z^{**}$ of
$\{ z^{(k)}|_{2t} \}_{k=d_2}^\infty$
satisfies $z^* = z^{**}|_{2d_1}$.
%%%%%%%%%%%%%%%%%%%%%%%%%%%%%%%%%%%%%%%%%%%%%%%%%%%%%%%%%%%%%
\iffalse

if the minimizer $z^*$ of \reff{<Rz>:momopt}
has a unique representing measure supported in $K^\circ$,
then when $t$ is sufficiently large
\beq \label{FT:z**}
\rank \, M_t[z^{**}] \, = \, \rank \, M_{t-d_c}[z^{**}],
\eeq

\fi
%%%%%%%%%%%%%%%%%%%%%%%%%%%%%%%%%%%%%%%%%%%%%%%%%%%%%%%%%%%%
\eit
\end{theorem}
\proof
(i) As in the proof for item (i) of Theorem~\ref{thm:unbounded:pop},
the trace of $M_{d_1}[z]$ can be bounded by a constant.
Similarly, it implies that \reff{<Rz>:momopt} has a minimizer $z^*$.
The minimizer $z^*$ is unique, since the objective $\langle R, z\rangle$
is linear in $z$ and has generic coefficients. The membership
$z^* \in \mathscr{R}_{2d_1}( K^\circ )$ implies that
$z^*$ has a decomposition like \reff{atomic:z*}.
We only need to show that $r \le m$.
Consider the following linear program in $(\tau_1, \ldots, \tau_r)$:
\beq  \label{min:nu:z*dcomp:unbd}
\left\{ \baray{cl}
\min & \tau_1 R(u_1) + \cdots + \tau_r R(u_r)  \\
\st & -1 \ge \sum\limits_{j=1}^r \tau_j g_i(u_j), \, i= 1,\ldots, m,  \\
  & \tau_1 \ge 0, \ldots, \tau_r \ge 0.
\earay \right.
\eeq
Note that \reff{min<Ry>:mom} and
\reff{min:nu:z*dcomp:unbd} have the same optimal value.
Since it is a linear program, \reff{min:nu:z*dcomp:unbd}
has a minimizer $\tau^* = (\tau_1^*, \ldots, \tau_r^*)$
of at most $m$ nonzero entries (see \cite{Bertsimas1997linear}).
This implies that the number $r$ in \reff{atomic:z*}
can be chosen to be at most $m$.

(ii) Since $R$ lies in the interior of $\Sig[x]_{2d_1}$,
there is $\eps >0$ such that $R - \eps \in \Sig[x]_{2d_1}$.
Then the constraint $M_k[z] \succeq 0$ implies that
\[
\langle R, z \rangle - \eps (z)_0 =
\langle R - \eps, z \rangle \ge 0.
\]
So we get that
$(z)_0 \le \frac{\langle R, z \rangle}{ \eps}$.
The optimal value of \reff{min<Ry>:mom} is always less than or equal to
that of \reff{min<R,z>:gz=-1}.
Therefore, the sequence $\{ \big( z^{(k)} \big)_0 \}_{k=d_2}^\infty$
is bounded. Moreover, the constraint $L^{(k)}_{x^Tx-1}[z] = 0$
implies that
\[
(z)_{2\af} = (z)_{2e_1+2\af} + \cdots + (z)_{2e_n+2\af}
\ge \max \big((z)_{2e_1+2\af}, \ldots,  (z)_{2e_n+2\af} \big)
\]
for all monomial powers $\af$. The diagonal entries of
the psd moment matrix $M_k[z]$ are precisely the entries
$(z)_{2\bt}$, for powers $\bt$. This implies that the sequence
$\{ \big( z^{(k)} \big)_{2\bt} \}_{k=d_2}^\infty$
is bounded for all powers $\bt$.
Therefore, for each fixed $t \ge d_1$,
the sequence of each diagonal entry of $M_t[z^{(k)}]$ is bounded,
and so is the truncated sequence $\{ z^{(k)}|_{2t} \}_{k=d_2}^\infty$.
Let $\mcal{H}_k$ be the set of feasible points $z$ in
\reff{min<Ry>:mom} for the relaxation order $k$,
except the first $m$ inequalities. Denote the truncation:
\[
G_k \,  \coloneqq  \, \{ z|_{2t}: \, z \in \mcal{H}_k \}.
\]
Then, $G_{k+1} \subseteq G_k$ for all $k$.
Since there is a sphere constraint $x^Tx=1$,
the quadratic module for the set $K^\circ$ is archimedean,
so (see Prop.~3.3 of \cite{linmomopt})
\[
\mathscr{R}_{2t}(K^\circ) = \bigcap_{k=d_2}^\infty G_k .
\]
If $z^{**}$ is an accumulation point of $\{ z^{(k)}|_{2t} \}_{k=d_2}^\infty$,
then $z^{**} \in G_k$ for all $k$ and hence
$z^{**} \in \mathscr{R}_{2t}(K^\circ)$.
Note that the truncation $z^{**}|_{2d_1}$
is also a minimizer of \reff{<Rz>:momopt}.
Since the minimizer is unique,
we must have $z^* = z^{**}|_{2d_1}$.
%By the assumption, $z^{**}|_{2d_1}$ has a unique representing measure.
%By the item (i),  $z^{**}|_{2d_1}$ has a unique decomposition as in \reff{atomic:z*}.
%Note that every representing measure for $z^{**}$
%must also be a representing measure for the truncation $z^{**}|_{2d_1}$.
%Therefore, when $t$ is large enough,
%the rank condition \reff{FT:z**} must be satisfied.
% \Halmos
\endproof

The optimization \reff{<Rz>:momopt} is a linear conic optimization problem
with the moment cone. It can also be viewed as a generalized moment problem.
When the constraining set is compact,
we refer to \cite{Las08,linmomopt} for how to solve it;
when the set is unbounded, we refer to the recent work
\cite{HNY1,HNY2}.

% \setcitestyle{numbers}
% \bibliographystyle{informs2014}
% \bibliographystyle{plainnat}
% \bibliographystyle{amsplain}
% \bibliography{multi_ref.bib}

% \begin{thebibliography}{100}

%%%%%%%%%%%%%%%% Author Year format

\end{document}